\documentclass[12pt]{article}
\usepackage{graphicx}  
\usepackage{amssymb, amsthm, amsmath, enumerate}
\usepackage{titlesec}
\usepackage{dsfont}
\usepackage[parfill]{parskip}

\usepackage[a4paper, margin=1in]{geometry}
\usepackage{mathtools}
\usepackage{bbm}

\usepackage{hyperref}
\hypersetup{
    colorlinks=true,
    linkcolor=blue,
    filecolor=magenta,      
    urlcolor=cyan,
    pdftitle={Overleaf Example},
    pdfpagemode=FullScreen,
}
    
\newcommand{\R}{\mathbb{R}}

\newcommand{\dist}{\operatorname{dist}}
\newcommand{\Det}{\operatorname{Det}}

\theoremstyle{plain}
\newtheorem{theorem}{Theorem}[section]
\newtheorem{lemma}[theorem]{Lemma}
\newtheorem{corollary}[theorem]{Corollary}
\newtheorem{proposition}[theorem]{Proposition}
\newtheorem{definition}[theorem]{Definition}

\newtheorem{remark}[theorem]{Remark}
\numberwithin{equation}{section}

\title{Cartan's and Gauss's equations and rigidity theorems for isometric embeddings in low Sobolev regularity}
\author{Isaac Newell\thanks{Mathematical Institute and Hertford College, University of Oxford, Andrew Wiles Building, Radcliffe Observatory Quarter, Woodstock Road, Oxford OX2 6GG, UK. Email: isaac.newell@maths.ox.ac.uk.}~\thanks{Supported by a Clarendon Scholarship from the University of Oxford.} ~ and Luc Nguyen\thanks{Mathematical Institute and St Edmund Hall, University of Oxford, Andrew Wiles Building, Radcliffe Observatory Quarter, Woodstock Road, Oxford OX2 6GG, UK. Email: luc.nguyen@maths.ox.ac.uk.}}
\date{}

\begin{document}

\maketitle
\begin{abstract}
    Let $\{\eta^i\}_{i=1}^2$ be a an orthonormal coframe on a domain $U$ on a smooth surface $(\Sigma,g)$. When $\eta^i$ is smooth, it is well-known that there is a unique connection 1-form $\omega$ verifying Cartan's first structural equations  $d\eta^i = (*\eta^i) \wedge \omega$, and Cartan's second structural equation $d\omega = K_g dvol_g$. We prove that this statement remains valid when the frame is $C^0 \cap H^{\frac12}$, where the structural equations are understood in the sense of distributions. From this, we deduce that the Gauss equation $\Det D^2 f = K_g (1+|Df|^2)^2$ holds for every graphical representation $f$ of an isometric embedding of regularity $C^1 \cap W^{1+\frac23,3}$ or $c^{1,\frac12} \cap BV^2$. As an application, we prove regularity and convexity results for isometric embeddings of closed surfaces and convex caps with $K_g \geq 0$. 
    
\end{abstract}

\setcounter{tocdepth}{1}
\tableofcontents

\section{Introduction}
In the local study of immersed surfaces in $\R^3$, various formalisms are available to express the relationship between intrinsic and extrinsic geometries. Notably, the equations of Cartan, Codazzi, Darboux and Gauss  play a fundamental role. In this paper, we study the equations of Cartan and Gauss for isometric embeddings with low Sobolev regularity.

Let us begin by recalling these equations in the smooth setting. Suppose that $(\Sigma,g)$ is a smooth, oriented Riemannian surface. Its local intrinsic geometry is described in a local orthonormal frame via the Cartan structural equations. Let $\{e_i\}_{i=1}^2 \subset \Gamma(U;T\Sigma)$ be a smooth orthonormal frame on a domain $U \subset \Sigma$, $\eta^i := g(\cdot, e_i)$ the dual coframe, and let $\nabla$ denote the Levi-Civita connection of $(\Sigma,g)$. Then, there is a unique connection 1-form $\omega(X) := g(\nabla_X e_2,e_1)$ satisfying the first structural equations
\begin{equation} \label{eq:cartan1}
    \begin{cases}
        d\eta^1 = \eta^2 \wedge \omega,\\
        d\eta^2 = -\eta^1 \wedge \omega,
    \end{cases}
\end{equation}
and the second structural equation
\begin{equation} \label{eq:cartan2}
    d\omega = K_g ~dvol_g,
\end{equation}
where $K_g$ and $dvol_g$ denote the Gaussian curvature and Riemannian volume form of $(\Sigma,g)$. We also note that the connection 1-form can be defined using the interior product:
\begin{equation} \label{eq:conn_form_formula}
    \omega = \iota_{e_2}d\eta^1 - \iota_{e_1}d\eta^2.
\end{equation}
For a more detailed account of Cartan's formalism, see Clelland \cite{clelland_17}.

Suppose further that $(\Sigma, g)$ is isometrically embedded in $\R^3$. The extrinsic geometry of this embedding is governed by the Gauss and Codazzi equations. In this paper we consider only the Gauss equation, with respect to a local graphical representation. We assume that $u : U \subset \R^2 \to \R^3$ is a smooth isometric embedding of a smooth metric $g$ such that $u = (\Psi,v)$ where $\Psi = (u^1,u^2) : U \to U'$ is a diffeomorphism with inverse $\Phi$ onto an open set $U' \subset \R^2$. If $f := v \circ \Phi$, then
\begin{equation*}
    u(U) = \{(x,f(x)) : x = (x_1,x_2) \in U'\}
\end{equation*}
and the Gauss equation reads
\begin{equation} \label{eq:gauss1}
    \det D^2 f = K_g(\Phi)(1+|Df|^2)^2 \quad \text{in } U'.
\end{equation}
It is a simple fact that, in this smooth setting, the Cartan equation \eqref{eq:cartan2} and the Gauss equation \eqref{eq:gauss1} are equivalent.

In recent years, there has been a growing interest in the study of isometric embeddings of low regularity. For a non-exhaustive list of references, see, e.g.,  Cao -- Hirsch -- Inauen \cite{cao_hirsch_inauen_Darboux_25}, Cao -- Inauen \cite{cao_inauen_2024}, Cao -- Székelyhidi \cite{cao_sz_isoExt_19, cao_sz_2022, cao_szekelyhidi_25}, Chen -- Li \cite{chen_li_18, chen_li_2021}, Conti -- De Lellis -- Székelyhidi \cite{conti_deLellis_szekelyhidi_12}, De Lellis -- Inauen \cite{deLellis_inauen_20}, De Lellis -- Székelyhidi -- Inauen \cite{dl_sz_in_2018}, De Lellis -- Pakzad \cite{deLellis_pakzad_21}, Giron \cite{giron_21}, Inauen \cite{inauen_19}, Li -- Pakzad -- Schikorra \cite{li_pakzad_schikorra_21}, Liu -- Mal\'{y} \cite{liu_maly}, Müller -- Pakzad \cite{muller_pakzad_05}, Pakzad \cite{pakzad_04, pakzad_24}. These works naturally motivate the following two basic questions:
\begin{enumerate}
    \item Do Cartan's structural equations \eqref{eq:cartan1} and \eqref{eq:cartan2} hold for less regular frames?
    \item Does the Gauss equation \eqref{eq:gauss1} hold for less regular isometric embeddings $u$?
\end{enumerate}
We note that in the works cited above, the Hessian of $f$ does not make sense classically and the connection 1-form $\omega$ was not an object of study except for in \cite{giron_21}. In this paper, we restrict ourselves to the situation where the connection 1-form $\omega$ and the Hessian determinant of $f$ make sense as distributions. 

For both questions, the answer is false if the regularity is allowed to be too low, as simple examples illustrate. Consider the following example for the Cartan structural system. Take $U := B \subset \R^2$ to be the unit ball, with $g = dx_1^2 + dx_2^2 = dr^2 + r^2 d\theta^2$ the Euclidean metric, and the radial frame
\begin{equation*}
    \eta^1 := dr, \quad \eta^2 := rd\theta,
\end{equation*}
which belongs to $W^{1,p} \cap L^\infty(B;T^*\R^2)$ for any $1 \leq p < 2$. There is a unique connection 1-form
\begin{equation*}
    \omega = -d\theta \in L^p(B; T^* \R^2)\quad \forall 1\leq p < 2
\end{equation*}
verifying the first structural equations \eqref{eq:cartan1} in the $L^p$ sense. However, there is a defect in the second structural equation \eqref{eq:cartan2}: in sense of distributions,
\begin{equation*}
    d\omega = -2\pi \delta_0 \neq K_g ~dvol_g = 0.
\end{equation*}
As for the Gauss equation, one may consider, for example, maps of the form
\begin{equation} \label{eq:liu_maly_cone}
    u(r\cos\theta, r\sin\theta) := r\gamma(\theta)
\end{equation}
where $\gamma : [0,2\pi] \to \mathbb{S}^2$ is a smooth closed simple curve parametrised by arc length. Then, $u$ belongs to $W^{2,p} \cap C^{0,1}(B;\R^3)$ for any $1 \leq p < 2$, is smooth away from $0$, and is an isometric embedding of the Euclidean unit disc into $\R^3$. However, $\gamma$ may be chosen such that $u(B)$ is the graph of a $W^{2,p} \cap C^{0,1}$ function $f$ defined on a planar domain and $\Det D^2 f$ is a Dirac mass, and hence the Gauss equation is false -- see Appendix \ref{sec:examples} for the computation. (Similar examples were considered in \cite{liu_maly} for a closely related equation, namely the Darboux equation, which we do not study in this paper.) Therefore, any extension of Cartan's or Gauss's equation must impose minimal but nontrivial Sobolev regularity.

Our first main result, Theorem \ref{thm:cartan_main}, extends the Cartan formalism \eqref{eq:cartan1} - \eqref{eq:cartan2} to frames of regularity $C^0 \cap H^{\frac12}$, a Sobolev threshold already known to be optimal in related Jacobian problems \cite{brezis_nguyen_11}. This extends the $W^{1,p}$ case treated by Giron \cite{giron_21}.

\begin{theorem} \label{thm:cartan_main}
    Let $(\Sigma,g)$ be a compact Riemannian surface and $U \subseteq \Sigma$ be a smooth orientable open subset endowed with volume form $dvol_g$. Suppose that $\{\eta^i\}_{i=1}^2 \subset C^0\cap H^\frac12_{loc}(U;T^*\Sigma)$ is a positively-oriented $g$-orthonormal coframe with dual frame $\{e_i\}_{i=1}^2$. Define the distributional $1$-form $\omega$ by \eqref{eq:conn_form_formula} via Proposition \ref{prop:WeakProducts}. Then, $\omega$ belongs to $H^{-\frac12}_{loc}(U; T^*\Sigma)$, is the unique solution in $H^{-\frac12}_{loc}(U;T^*\Sigma)$ to \eqref{eq:cartan1} in the sense of distributions via Proposition \ref{prop:WeakProducts}, and satisfies \eqref{eq:cartan2} in the sense of distributions:
    \begin{equation*}
        d\omega [\psi] = \omega[d\psi] = \int_{U} K_g \psi ~dvol_g \quad \forall \psi \in C^\infty_c(U).
    \end{equation*}
\end{theorem}

Before passing onto the next results, we briefly recall the notion of distributional Jacobian and Hessian determinant. A theorem from Brezis -- Nguyen \cite{brezis_nguyen_11} says that, if $U \subset \R^2$ is a bounded Lipschitz domain and $w \in H^\frac12(U;\R^2)$, then $\Det Dw$ is a well-defined distribution. (See also Ball \cite{ball_77}, Morrey \cite{morrey}, and Reshetnyak \cite{reshetnyak67} for earlier work.) As a consequence, for $f \in H^\frac32(U)$, the distributional Hessian $\Det D^2f$ is defined as the distributional Jacobian of $Df$. It is a fact that this definition coincides with
\begin{equation*}
    \Det D^2 f = (f_{x_1}f_{x_2})_{x_1 x_2} - \frac12 (f_{x_1}^2)_{x_2 x_2} - \frac12 (f_{x_2}^2)_{x_1 x_1},
\end{equation*}
which makes sense for even less regular $f$.

The following Corollary \ref{cor:cartan_particular_frame} of Theorem \ref{thm:cartan_main} is a key tool in our passage from the Cartan equation \eqref{eq:cartan2} to the Gauss equation \eqref{eq:gauss1}.
\begin{corollary} \label{cor:cartan_particular_frame}
    Let $U \subset \R^2$ be a domain with a smooth Riemannian metric $g$. Suppose that $u \in C^1 \cap H^{\frac32}_{loc}(U;\R^3)$, is an isometric embedding of $(U,g)$ and that $u(U)$ is parameterised as a graph of $f: U' \to \R$, where $u = (\Psi,v)$ and $\Psi: U \to U'$ is a diffeomorphism with inverse $\Phi$, and $f := v \circ \Phi$. Then, $f$ belongs to $C^1 \cap H^{\frac32}_{loc}(U')$ and satisfies
    \begin{equation} \label{eq:jacobian_intermediate}
        \Det Dw= K_g(\Phi)(1+|Df|^2)^\frac12 \quad \text{in } \mathcal{D}'(U'),
    \end{equation}
    where
    \begin{equation*}
        w := F(Df), \quad F : \R^2 \to \R^2, \quad F(z):= \big((1+|z|^2)^{-\frac12}(1+z_2^2)^{-1} z_1, z_2\big).
    \end{equation*}
\end{corollary}

Equation \eqref{eq:jacobian_intermediate} is Cartan's second structural equation for a particular frame on $(U', g')$ with $g' = \Phi^*g$, but because $g'$ is not smooth we must pull back onto $(U,g)$ to recover \eqref{eq:jacobian_intermediate} from Theorem \ref{thm:cartan_main}. For a similar result for the Darboux equation, see the very recent preprint Cao -- Hirsch -- Inauen \cite{cao_hirsch_inauen_Darboux_25}.

Based on Corollary \ref{cor:cartan_particular_frame} we obtain our second main result, concerning the validity of the Gauss equation under a stronger regularity assumption. For the definitions of the spaces $c^{1,\frac12}$ and $BV^2$, see \S \ref{sec:func_spaces}.

\begin{theorem} \label{thm:Gauss_23}
    Let $U \subset \R^2$ be a domain with a smooth Riemannian metric $g$. Suppose that $u \in C^1(U;\R^3)$ is an isometric embedding of $(U,g)$ and that $u(U)$ is parametrised as a graph of $f: U' \to \R$, where $u = (\Psi,v)$ with $\Psi: U \to U'$ a diffeomorphism with inverse $\Phi$, and $f := v \circ \Phi$.
    \begin{enumerate}[(i)]
        \item \label{case:GaussBV} If $u \in c^{1,\frac12}_{loc} \cap BV^2_{loc}(U;\R^3)$, then $f \in c^{1,\frac12}_{loc} \cap BV^2_{loc}(U')$ and $f$ satisfies the Gauss equation in the sense of distributions, i.e.,
        \begin{equation} \label{eq:gauss2}
            \Det D^2 f = K_g(\Phi)(1+|Df|^2)^2 \quad \text{in } \mathcal{D}'(U').
        \end{equation}
        \item \label{case:Gauss23} Likewise, if $u \in C^1 \cap W^{1+\frac23,3}_{loc}(U;\R^3)$, then $f \in C^1 \cap W^{1+\frac23,3}_{loc}(U')$ and \eqref{eq:gauss2} holds.
    \end{enumerate}
\end{theorem}
Note that if $u$ is as in case \eqref{case:GaussBV}, then $u$ belongs to $C^1 \cap W^{1+\frac23,p}_{loc}(U;\R^3)$ for any $p < 3$ by interpolation. However, we note that $c^{1,\frac12} \cap BV^2$ does not embed into $W^{1+\frac23,3}$ in two dimensions; see Brezis -- Mironescu \cite[Theorem 1]{brezis_mironescu_18}.\footnote{The statement of \cite[Theorem 1]{brezis_mironescu_18} appears to give only a function $f \in (C^{1,\frac12} \cap W^{2,1}) \setminus W^{1+\frac23,3}$, however their construction in fact gives $f \in c^{1,\frac12}$.}

Combining Theorem \ref{thm:Gauss_23} with a $\det = \Det$ result of Fonseca and Mal\'{y} \cite[Theorem 1.4]{fonseca_maly} (see also Müller \cite{mueller_det}), we deduce the following.
\begin{corollary} \label{cor:Gauss_ae}
    Let $u$ be as in Theorem \ref{thm:Gauss_23}, case \eqref{case:GaussBV}. Denote by $\nabla^2f$ the absolutely continuous part of the Radon measure $D^2f$ with respect to 2-dimensional Lebesgue measure. Then
    \begin{equation*}
        \det(\nabla^2 f(x)) = K_g(\Phi(x))(1+|Df(x)|^2)^2 \quad \text{for}\quad \mathcal{L}^2 \text{-}a.e. ~ x \in U'.
    \end{equation*}
\end{corollary}

To place our work in context, let us discuss relevant literature in the setting of low Sobolev regularity. Regarding the validity of the Cartan structural system, the only earlier work is \cite{giron_21} where the frames are in $W^{1,p}$. It was shown that the Cartan structural equations are correct modulo possibly nontrivial defect measures, which vanish if the frame is also continuous. 

The Gauss equation in low regularity has received slightly more attention. First, suppose that $u \in W^{2,2}(U; \R^3)$. In this case, the second fundamental form of $u$ belongs to $L^2$ and the Gauss equation \eqref{eq:gauss1} holds in the $L^1$ sense, following from a simple approximation by smooth maps (see, e.g., \cite[Lemma 2.2]{hornung_18}). We point out that in the discussion at hand, we already assumed a local graphical representation of the embedding $u$. The existence of such graphical representation in $W^{2,2}$ regularity is a more subtle matter; see Hornung -- Vel\v{c}i\'{c} \cite{hornung_18}, and also Toro \cite{toro_94} and Müller -- \v{S}ver\'{a}k \cite{mueller_sverak_95}.

For isometric embeddings less than twice differentiable, two important works precursory to our Theorem \ref{thm:Gauss_23} are Conti -- De Lellis -- Székelyhidi \cite{conti_deLellis_szekelyhidi_12} and Pakzad \cite{pakzad_24}. If $\alpha > \frac23$, $(M,g)$ is a smooth surface, and $u \in C^{1,\alpha}(M;\R^3)$ an isometric embedding with Gauss map $\nu \in C^{0,\alpha}(M;\mathbb{S}^2)$, then in \cite[Proposition 6]{conti_deLellis_szekelyhidi_12} it is shown that
\begin{equation} \label{eq:delellis_gauss}
        \int_V \phi(\nu(x)) K_g(x)~dvol_g(x) = \int_{\mathbb{S}^2} \phi(y) \operatorname{deg}(\nu, V, y) ~d\sigma(y)
\end{equation}
for all $V \subset \subset M$ and $\phi \in L^\infty(\mathbb{S}^2)$ with $\operatorname{supp}\phi \subset \mathbb{S}^2\setminus \nu(\partial V)$. The formula \eqref{eq:delellis_gauss} says that, in a weak sense, $K_g$ is the Jacobian of the Gauss map. Equation \eqref{eq:delellis_gauss} is closely related to the Gauss equation \eqref{eq:gauss2} but we do not know whether they are equivalent. (However, see Proposition \ref{prop:CdLS_vs_Gauss} where we prove that \eqref{eq:delellis_gauss} holds under the hypothesis of case \eqref{case:Gauss23} of Theorem \ref{thm:Gauss_23}.) 

In $C^{1,\alpha}$ regularity with $\alpha > \frac23$, we note that the Gauss equation \eqref{eq:gauss2} can be deduced from the block of unnumbered equations at the end of the proof of Proposition 4.9 in the recent work of Pakzad \cite{pakzad_24}. However, our method differs from theirs -- those authors did not consider the Cartan structural equations, while we obtain the Gauss equation as a consequence thereof.

Now we turn to a discussion of the proof of Theorems \ref{thm:cartan_main} and \ref{thm:Gauss_23}. Our proof of Theorem \ref{thm:cartan_main} is based on the lifting property of continuous $\mathbb{S}^1$-valued maps, exploiting the 2-dimensional nature of our problem. Fix any smooth orthonormal coframe $\{\hat\eta^i\}_{i=1}^2$, with connection $1$-form $\hat\omega$. By the aforementioned lifting property, there exists $\theta \in H^{\frac12}(U) \cap C^0(\overline U)$ such that
\begin{align*}
    \eta^1 &= \cos\theta~ \hat\eta^1 + \sin\theta~ \hat\eta^2,\\
    \eta^2 &= -\sin\theta~\hat\eta^1 + \cos\theta~ \hat\eta^2.
\end{align*}
Let $\theta_\epsilon$ be a smooth approximation of $\theta$, and define $\eta^i_{(\epsilon)}$ like $\eta^i$ but with $\theta_\epsilon$ replacing $\theta$. Then, the smooth frames $\eta^i_{(\epsilon)}$ satisfy Cartan's structural equations with connection 1-form $\omega_{(\epsilon)} = \hat{\omega} - d\theta_\epsilon$. We pass to the limit $\epsilon \searrow 0$ via Proposition \ref{prop:WeakProducts} in the structural equations. In particular, $\omega = \hat\omega - d\theta$, so that the second structural equation follows from the identity $d^2 = 0$.

We comment next on the passage from Corollary \ref{cor:cartan_particular_frame} to Theorem \ref{thm:Gauss_23} and the jump in the regularity requirement for the embedding map $u$. We point out that Corollary \ref{cor:cartan_particular_frame} requires that $Du$ is fractionally differentiable of order $\frac12$, while Theorem \ref{thm:Gauss_23} requires fractional differentiability of order at least $\frac23$. If $f$ were smooth, then with $w = F(Df)$ as in Corollary \ref{cor:cartan_particular_frame}, there holds
\begin{equation} \label{eq:detDw_from_f}
    \det Dw = \det DF(Df) \det D^2f = (1+|Df|^2)^{-\frac32} \det D^2 f,
\end{equation}
and therefore, the Gauss equation \eqref{eq:gauss2} follows by multiplying \eqref{eq:jacobian_intermediate} by $(1+|Df|^2)^{\frac32}$. In Theorem \ref{thm:Gauss_23}, $f$ is less regular. When $Du$ is fractionally differentiable of order $\frac12$, we do not know yet how to justify this multiplication. A version of the chain rule for the Jacobian determinant in this regularity is available in Gladbach -- Olbermann \cite{gladbach_olbermann}, but as far as we can see, it appears unsuitable for our application at hand. Observe that, by standard Jacobian estimates, if $f \in W^{1+s,p} \cap C^1$, then $\Det D^2 f$ belongs to $(W^{2(1-s),\frac{p}{p-2}}_0)^*$ while $(1+|Df|^2)^{-\frac32}$ belongs to $W^{s,p}$. Hence, a natural assumption to justify the above multiplication procedure is that  $s = \frac23$ and $p=3$ as in case (ii) of Theorem \ref{thm:Gauss_23}. In case (i) it is not clear to us how to make such a duality argument, so we prove the Gauss equation from \eqref{eq:jacobian_intermediate} by a different approach. We prove a chain rule for the distributional Jacobian determinant, Lemma \ref{prop:BV_det_chain}, that suffices in our regularity class to recover the Gauss equation from \eqref{eq:jacobian_intermediate}. 

Regardless of whether or not the above multiplication procedure can be made sense of, it would be interesting to determine for which $s \in (0,1)$ and $p \in [1,\infty]$ the Gauss equation \eqref{eq:gauss2} holds for all $C^1$ isometric immersions $u \in W^{1+s,p}(U;\R^3)$. By Theorem \ref{thm:Gauss_23} and Sobolev embeddings, \eqref{eq:gauss2} is valid if $s=\frac23,p=3$ or $s > \frac23, p \geq \frac2s$. We do not yet know how to treat the cases $s=\frac23, p \neq 3$; $s> \frac23, p < \frac2s$; or $s < \frac23$. However, we give the following result on removability of small singular sets.

\begin{theorem} \label{thm:capacity}
Let $U \subset \R^2$ be a domain equipped with a smooth Riemannian metric $g$. Suppose that $u\in C^1 \cap W^{1+s,p}_{loc}(U;\R^3)$ is an isometric embedding of $g$ where $s \in (\frac12,1)$ and $p \geq \frac2s$, such that the embedded surface $u(U)$ is the graph of $f : U' \subset \R^2 \to \R^3$ as in Theorem \ref{thm:Gauss_23}. Let $s_0 = 2(1-s)$, $p_0 = \frac{p}{p-2}$ and suppose that there exists a compact set $S \subset U$ of $(s_0,p_0)$-capacity zero (see \eqref{eq:def_cap_wsp}) such that $u \in W^{1+\frac23,3}_{{loc}}(U \setminus S; \R^3)$. Then, the Gauss equation \eqref{eq:gauss2} holds.
\end{theorem}

Note that, since $sp \geq 2$, $s_0p_0 \leq 2$. In particular, since $S$ has $(s_0,p_0)$-capacity zero, the Hausdorff dimension of $S$ is at most $2 - s_0p_0$. (Note also that  if a set has finite $\mathcal{H}^{2-s_0p_0}$ measure, then it has $(s_0,p_0)$-capacity zero.)

Finally, we discuss two applications of Theorem \ref{thm:Gauss_23}. We recall that a celebrated theorem of Cohn-Vossen \cite{cohn_vossen} asserts that any $C^2$ isometric embedding into $\R^3$ of a given two-sphere with positive Gaussian curvature is convex and unique up to rigid motion -- see also Herglotz \cite{herglotz}. (Such rigidity stands in contrast to the Nash-Kuiper theorem \cite{nash_54, kuiper_55}.) Borisov \cite{borisov_I, borisov_II, borisov_59, borisov_III, borisov_60} extended Cohn-Vossen's theorem to $C^{1,\alpha}(\mathbb{S}^2;\R^3)$ isometric embeddings with $\alpha > \frac23$. A different proof of Borisov's result, based on \eqref{eq:delellis_gauss}, was later given by Conti -- De Lellis -- Székelyhidi \cite{conti_deLellis_szekelyhidi_12}. Subsequently, Pakzad \cite{pakzad_24} extended the result to include metrics of nonnegative curvature. By combining Theorem \ref{thm:Gauss_23} with the arguments in \cite{conti_deLellis_szekelyhidi_12, pakzad_24} we obtain the following extensions.

\begin{theorem} \label{thm:rigidity_closed}
    Let $g$ be a smooth metric on $\mathbb{S}^2$ with Gaussian curvature $K_g \geq 0$. If $u \in C^1 \cap W^{1+\frac23,3}(\mathbb{S}^2;\R^3)$ is an isometric embedding of $g$, then the image surface $\Sigma := u(\mathbb{S}^2)$ is a closed convex surface. Moreover, any two such images are congruent in $\R^3$. Finally, if $K_g > 0$ then $\Sigma$ is smooth.
\end{theorem}

Parallel to Theorem \ref{thm:rigidity_closed}, we also prove the following rigidity\footnote{See also De Lellis and Inauen \cite{deLellis_inauen_20} for a related rigidity statement for the Levi-Civita connection for embedded polar caps.} and regularity result for convex caps (compare with Han and Hong \cite[Theorem 8.1.7]{han_hong_06} in the smooth case). 
\begin{theorem} \label{thm:rigidity_caps}
     Let $U \subset \R^2$ be a bounded domain equipped with a Riemannian metric $g$ smooth up to the boundary of $U$. Suppose that $u \in W^{1+\frac23,3}_{loc}(U;\R^3) \cap C^1(\overline U;\R^3)$ is an isometric embedding of $(\overline U, g)$ into $\R^3$. Assume that there is a plane $\Pi \subset \R^3$ such that $u(\overline U)$ is a graph over $\Pi$, $u(\partial U) \subset \Pi$, and that there is a component $\gamma_1$ of $\partial U$ which has a neighbourhood $\mathcal{N} := \{x \in U : 0 < \dist(x,\gamma_1) < \epsilon\}$ such that $u(\mathcal{N})$ lies strictly on one side of $\Pi$. If $K_g \geq 0$, then $u(\overline U)$ is a convex surface whose projection $U'$ onto $\Pi$ is a convex domain. Furthermore, if $K_g > 0$, then $u \in C^\infty(U;\R^3)$.
\end{theorem}

The remainder of the paper is organized as follows. In \S\ref{sec:prelims} we recall some preliminaries on function spaces, distributional products, and lifting. In \S\ref{sec:proof_cartan}, we prove Theorem \ref{thm:cartan_main} using the tools from \S\ref{sec:prelims}. In \S\ref{sec:gauss}, we prove the Gauss equation, Theorem \ref{thm:Gauss_23}, via the Cartan framework. We also study removable singularities for the Gauss equation and prove Theorem \ref{thm:capacity}. Finally, in \S\ref{sec:positive} we use Theorem \ref{thm:Gauss_23} and Pogorelov's theory of surfaces of bounded extrinsic curvature to prove Theorems \ref{thm:rigidity_closed} and \ref{thm:rigidity_caps}.

\section{Preliminaries} \label{sec:prelims}
    As preparation for the proof of Theorem \ref{thm:cartan_main} (given in \S \ref{sec:proof_cartan}) we introduce some preliminaries. In \S \ref{sec:func_spaces}, we recall some basic facts about fractional Sobolev and little Hölder spaces. We introduce these spaces for tensor fields as well as functions. In \S \ref{sec:forms_ops} we discuss operations on distributional differential forms, which will enable us to properly interpret Cartan's structural equations for frames of the regularity considered in Theorem \ref{thm:cartan_main}. In \S \ref{sec:lifting} we review the lifting of $\mathbb{S}^1$-valued maps. Apart from Proposition \ref{prop:WeakProducts} and Lemma \ref{lemma:pullback_commutation}, which contain the key technical tools used later, most of the material in this section is standard and is included to establish important notation and conventions. 
    
    \subsection{Spaces of functions and tensors} \label{sec:func_spaces}
    
    We begin by briefly defining relevant Sobolev spaces of real order for tensor fields on Riemannian manfiolds. We will assume the usual definitions for spaces of integer order and focus on those of non-integer order. All of this has appeared elsewhere, but we present it for completeness.
    
    Let $(M,g)$ be a smooth, compact Riemannian manifold of dimension $n$ and $U \subseteq M$ a smooth open subset. For $s \in (0,1)$ and $p \in [1,\infty)$, define the Sobolev space $W^{s,p}(U)$ as the completion of $C^\infty(\overline U)$ with respect to the norm
    \begin{equation*}
        \|f\|_{W^{s,p}(U)} := \|f\|_{L^p(U)} + [f]_{W^{s,p}(U)},
    \end{equation*}
    where
    \begin{equation*}
        \|f\|_{L^p(U)} := \Big( \int_U |f|^p dvol_g \Big)^{\frac1p}
    \end{equation*}
    and the Gagliardo seminorm is defined via the Riemannian distance $d$ of $(M,g)$:
    \begin{equation} \label{eq:gag_for_fns_on_mfd}
        [f]_{W^{s,p}(U)} := \Big( \int_U \int_U \frac{|f(x)-f(y)|^p}{d(x,y)^{n+sp}}\Big)^{\frac1p}.
    \end{equation}

    Note that one cannot replace the function $f$ in \eqref{eq:gag_for_fns_on_mfd} by a tensor field as one cannot directly take the difference of a tensor field at two different points. Instead one needs to work in a local coordinate system where tensor fields are trivialised. This is justified by the following observation.

    \begin{lemma}
    \label{lem:WspNE}
        Let $(M,g)$ be a compact Riemannian manifold of dimension $n$ and $U \subseteq M$ a smooth open subset. Let $r_0 \in (0,\operatorname{inj}(M,g))$ and fix a finite covering $\{B(x_i, r_0/2)\}_{i=1}^N$ of $M$ by balls. Let $U_i := U \cap B(x_i, r_0)$. Then, for any $s \in (0,1)$, $p \in [1,\infty)$ there exist constants $C_1, C_2 > 0$ depending on $s,p, (M, g), r_0, N$ such that for any $f \in C^\infty(\overline U)$,
    \begin{equation*}
       C_1 \|f\|_{W^{s,p}(U)} \leq  \sum_{i=1}^N \|f\|_{W^{s,p}(U_i)} \leq C_2 \|f\|_{W^{s,p}(U)}.
    \end{equation*}
    \end{lemma}

    \begin{proof}
        The second inequality is clear with $C_2 :=N$. For the first inequality, we let $S := \{(x,y)  \in U \times U : d(x,y) < \frac{r_0}{2}\}$ and estimate:
        \begin{equation} \label{eq:gagl_f_notS}
            \iint_{(U\times U)\setminus S} \frac{|f(x)-f(y)|^p}{d(x,y)^{n+sp}} dvol_g(x) dvol_g(y) \leq \Big(\frac{r_0}2\Big)^{-(n+sp)} 2^{p}~ (vol_g(U))^2 \|f\|_{L^p(U)}^p,
        \end{equation}
        and, since $S \subset \cup_{i=1}^N U_i \times U_i$,
        \begin{equation} \label{eq:gagl_f_S}
            \iint_{S} \frac{|f(x)-f(y)|^p}{d(x,y)^{n+sp}} dvol_g(x) dvol_g(y) \leq \sum_{i=1}^N [f]_{W^{s,p}(U_i)}^p.
        \end{equation}
        Adding \eqref{eq:gagl_f_notS} and \eqref{eq:gagl_f_S} yields the result.
    \end{proof}
    By Lemma \ref{lem:WspNE}, the norm on $W^{s,p}(U)$ can be equivalently defined using the sum of the norms on the charts $U_i$. It is also clear that any two norms defined using two different such coverings are equivalent. This motivates the following definition of fractional Sobolev spaces of tensor fields. Let $T^l_k M := TM^{\otimes l} \otimes T^*M^{\otimes k}$ denote the $(l,k)$-tensor bundle of $M$, where $l,k \in \mathbb{N}$. Let $r_0$ be the injectivity radius of $(M,g)$. Cover $M$ by finitely many geodesic balls of radius $\frac{r_0}2$, centred at $x_i$, $i=1,\dots,N$. Let $U_i := U \cap B(x_i,{r_0})$. On each $U_i$, fix a preferred coordinate system (e.g. the geodesic normal coordinate system at $x_i$). Then, we define the seminorm $[\cdot]_{W^{s,p}(U;T^l_k M)}$ on smooth $(l,k)$-tensor fields $\mu \in C^\infty(\overline U; T^l_k M)$ by
    \begin{equation} \label{eq:gag_sum_over_charts}
        [\mu]_{W^{s,p}(U;T^l_k M)} := \sum_{i=1}^N [\mu]_{W^{s,p}(U_i;T^l_k M)},
    \end{equation}
    where, on each chart, the seminorm is defined with respect to the chosen coordinate system: 
    \begin{equation*}
        [\mu]_{W^{s,p}(U_i; T^l_k M)} := \Big( \sum_{\alpha, \beta}\int_{U_i} \int_{U_i} \frac{|\mu^{\alpha_1, \dots, \alpha_l}_{\beta_1, \dots, \beta_k}(x) - \mu^{\alpha_1, \dots, \alpha_l}_{\beta_1, \dots, \beta_k}(y)|^p}{|x-y|^{n+sp}} dxdy\Big)^{\frac1p}.
    \end{equation*}
    The space $W^{s,p}(U;T^l_k M)$ (respectively $W^{s,p}_0(U;T^l_k M)$) is then defined as the completion of the space of smooth sections $C^\infty(\overline U;T^l_k M)$ (respectively $C^\infty_c(U;T^l_k M)$) with respect to the norm $\|\cdot \|_{W^{s,p}(U;T^l_k M)} := \|\cdot \|_{L^p(U;T^l_k M)} + [\cdot]_{W^{s,p}(U,T^l_k M)}$. As pointed out above, if one changes the atlas of charts then one gets an equivalent norm.

    Higher-order Sobolev spaces of functions and tensor fields can be defined via the covariant derivative. Let $k, l \in \mathbb{N}$. Assume that $s = m + \sigma$ where $m \in \mathbb{N}$ and $\sigma \in (0,1)$, and $p \in [1,\infty)$. Then, for $\mu \in C^\infty(\overline U;T^l_k M)$, the norm $\|\cdot\|_{W^{s,p}(U;T^l_k M)}$ is defined by
    \begin{equation*}
        \|\mu\|_{W^{s,p}(U;T^l_k M)} := \|\mu\|_{W^{m,p}(U;T^l_k M)} + [\nabla^m\mu]_{W^{\sigma,p}(U; T^l_{k+m}M)}.
    \end{equation*}
    The space $W^{s,p}(U;T^l_k M)$ (respectively $W^{s,p}_0(U;T^l_k M)$) is then defined as the completion of $C^\infty(\overline U; T^l_k M)$ (respectively $C^\infty_c(U;T^l_k M)$) with respect to the norm $\|\cdot\|_{W^{s,p}(U; T^l_k M)}$.

    We adopt the notation $W^{s,p}(U;\Lambda^kT^*M)$ (respectively, $W_0^{s,p}(U;\Lambda^kT^*M)$) for the subspace of $W^{s,p}(U;T^*M^{\otimes k})$ (respectively, $W_0^{s,p}(U;T^*M^{\otimes k})$) consisting of $W^{s,p}$ $k$-forms (that is, alternating $(0,k)$-tensors). Let $W^{-s,p'}(U;\Lambda^kT^*M)$ with $\frac{1}{p} + \frac{1}{p'} = 1$ denote the dual of $W_0^{s,p}(U;\Lambda^{n-k}T^*M)$. To describe this dual space, we use the framework of currents from geometric measure theory (see, e.g. \cite{simon}). Recall that, for $1\leq k \leq n$, a $k$-dimensional current in $U$ is a continuous linear functional on the space $\mathcal{D}^k(U)$ of infinitely-differentiable $k$-forms with compact support in $U$, with respect to its usual locally convex topology. The space of $k$-dimensional currents in $U$ is denoted $\mathcal{D}_k(U)$. For $s \in (0,1)$, $p \in [1,\infty)$, and $1 \leq k \leq n$, the space $W^{-s,p}(U;\Lambda^k T^*M)$ is defined as the dual of $W^{s,p}_0(U;\Lambda^{n-k}T^*M)$. It is clear that $W^{-s,p'}(U;\Lambda^{k}T^* \R^n)$ is a subspace of $\mathcal{D}_{n-k}(U)$. It is a fact that $\mathcal{D}^{k}(U)$ is dense in $W^{-s,p'}(U;\Lambda^{k}T^* M)$, where a $k$-form $\alpha \in \mathcal{D}^{k}(U)$ acts on $W^{s,p}_0(U;\Lambda^{n-k} T^* M)$ via the wedge product:
    \begin{equation*}
        \alpha[\phi] := \int_U \alpha \wedge \phi, \quad \phi \in W^{s,p}_0(U;\Lambda^{n-k} T^* M).
    \end{equation*}
    
    Some of our analysis involves continuous maps with Sobolev regularity, taking values in a bundle $E$ over $M$ (which will be either $\R$, $TM$ or $\Lambda^k T^*M$). To lighten the notation in such cases, we make a couple of definitions. With $E$ as just mentioned, we define
    \begin{equation} \label{eq:def_CH12}
    CH^\frac12(U;E) := C^0(\overline U;E) \cap H^\frac12(U;E)
    \end{equation}
    with the norm $\|\cdot\|_{CH^\frac12(U;E)} := \|\cdot\|_{C^0(\overline U;E)} + \|\cdot\|_{H^\frac12(U;E)}$. We also define
    \begin{equation} \label{eq:def_CH12_0}
            CH_0^\frac12(U;E) := \text{ closure of } C^\infty_c(U;E) \text{ in } CH(U;E).
    \end{equation}
    The dual of $CH_0^\frac12(U,\Lambda^k T^* M)$ is denoted by $CH^{-\frac12}(U;\Lambda^{n-k} T^*M)$. The spaces $CH^{\pm\frac12}(U;\R)$ and $CH_0^{\frac12}(U;\R)$ will be simply written as $CH^{\pm\frac12}(U)$ and $CH_0^{\frac12}(U)$, respectively.

    Returning now to spaces of real-valued functions on domains of $\R^n$, we shall also work with H\"{o}lder spaces $C^{k,\alpha}$, which we identify with $W^{k+\alpha,\infty}$, and their subspaces $c^{k,\alpha}$, for $k \in \mathbb{N}$ and $\alpha \in (0,1)$. Let $U \subset \R^n$ be a bounded Lipschitz domain. For $k\in \mathbb{N}$ and $\alpha \in (0,1)$ we denote
    \begin{equation*}
        C^{k,\alpha}(U) := \{f \in C^k(\overline U) : \|f\|_{C^{k,\alpha}(U)} := \|f\|_{C^k(\overline U)} + [D^k f]_{0,\alpha, U } < \infty\}
    \end{equation*}
    where
    \begin{equation*}
        [f]_{0,\alpha,U} := \sup_{x,y \in U,~x \neq y} \frac{|f(x)-f(y)|}{|x-y|^\alpha}.
    \end{equation*}
    We will also use the notation
    \begin{equation*}
        [f]_{0,\alpha,U|r} := \sup_{x,y \in U,~0 < |x-y| \leq r} \frac{|f(x)-f(y)|}{|x-y|^\alpha}.
    \end{equation*}
    The little Hölder space $c^{k,\alpha}(U)$ is defined to be the closure of $C^\infty(\overline U)$ in the $C^{k,\alpha}(U)$-norm.
    Furthermore, we say that $f \in C^{k,\alpha}_{loc}(U)$ (respectively, $f \in c^{k,\alpha}_{loc}(U)$) if $f \in C^{k,\alpha}(V)$ (respectively, $f \in c^{k,\alpha}(V)$) for any $V \subset\subset U$. It is a fact that, if $0 < \alpha < \beta < 1$ and $k \in \mathbb{N}$, then $C^{k,\beta}(U) \subsetneq c^{k,\alpha}(U) \subsetneq C^{k,\alpha}(U)$.
    
    We will also refer to the space $\mathcal{M}(U)$ of signed Radon measures on $U$ with finite mass and its subspace $\mathcal{M}^+(U)$ of positive Radon measures of finite mass. Furthermore, we denote
    \begin{equation*}
        BV^k(U) := \{u \in W^{k-1,1}(U): \partial^\beta u \in \mathcal{M}(U) \text{ for all multi-indices } \beta \text{ with } |\beta| = k \}
    \end{equation*}
    where $k \in \mathbb{N}^*$. Denote the seminorm
    \begin{equation*}
        [u]_{BV^k(U)} := \sum_{|\beta| = k} |\partial^\beta u|(U)
    \end{equation*}
    where $|\partial^\beta u|(U)$ is the total variation of the measure $\partial^\beta u$ on $U$. When $k=1$ we simply write $BV(U) := BV^1(U)$. 
    
    We recall the following simple estimates on the convolutions $f_\epsilon := f *\rho_\epsilon$ with a standard mollifier; see, e.g., \cite[Lemma 1]{conti_deLellis_szekelyhidi_12} for a proof.
    \begin{lemma} \label{lemma:c0alpha_convolution}
        Let $U \subset \R^n$ be a bounded Lipschitz domain, $V \subset\subset U$, and $f,g \in C^{0,\alpha}_{loc}(U)$. Let $\epsilon_0 :=\frac12\operatorname{dist}(V,\partial U) > 0$ and $W := V + B_{\epsilon_0} \subset\subset U$. Let $\rho \in C^\infty_c(B)$ be a standard mollifier kernel, and define $\rho_\epsilon(x):= \epsilon^{-n}\rho(x/\epsilon)$ and $f_\epsilon(x) := (f * \rho_\epsilon)(x)$, for $\dist(x,\partial U) < \epsilon$. Then, for $\epsilon < \epsilon_0$, the following hold.
        \begin{enumerate}
            \item[(i)] $\|f_\epsilon - f\|_{0;V} \leq c[f]_{0,\alpha;W|\epsilon} ~ \epsilon^\alpha$.
            \item[(ii)] $\|Df_\epsilon\|_{0,V} \leq c[f]_{0,\alpha;W|\epsilon}~\epsilon^{\alpha-1}$.
        \end{enumerate}
    \end{lemma}

    The following characterisation of little Hölder spaces is well-known; for a proof, see, e.g., \cite[Proposition 3.4]{deLellis_pakzad_21}.
    \begin{lemma}  \label{lemma:c0alpha_equivalence}
        Let $U \subset \R^n$ be a bounded Lipschitz domain and $\alpha \in (0,1)$, and let $f \in C^{0,\alpha}(U)$. Then, the following statements are equivalent.
        \begin{enumerate}
            \item[(i)] $f \in c^{0,\alpha}(U)$.
            \item[(ii)] $f: U \to \R$ admits an extension $\tilde{f}: \R^n \to \R$ such that $\tilde{f}_\epsilon \to f$ in $C^{0,\alpha}(U)$, where $\tilde{f}_\epsilon := \tilde{f} * \rho_\epsilon$ is the convolution with a standard mollifier $\rho_\epsilon(x) := \epsilon^{-n} \rho(x/\epsilon)$, $\rho \in C^\infty_c(B)$.
            \item[(iii)] $[f]_{0,\alpha,U|\epsilon} = o(1)$ as $\epsilon \searrow 0$.
        \end{enumerate}
    \end{lemma}

    We next state some simple lemmas on composition, inverses, and products in these spaces.

    Observe that if $f : \R^m \to \R$ is Lipschitz, $0 < s,\alpha < 1$, $1 \le p < \infty$, then $u \mapsto f(u)$ is a map from $W^{s,p}(U;\R^m)$ (respectively $c^{0,\alpha}(U;\R^m)$ and $BV(U;\R^m)$) into $W^{s,p}(U)$ (respectively $c^{0,\alpha}(U)$ and $BV(U)$), where we have used Lemma \ref{lemma:c0alpha_equivalence} for the $c^{0,\alpha}$ case. In the $W^{s,p}$ case we have the following continuity result. For a proof in the case $p=2$, see \cite[Proof of (5.43)]{bbm_04}. (See also the more general statement \cite[Theorem 15.6]{brezis_mironescu_21}.) Because the proof can be directly adapted to general $p$, we omit it here.
    
    \begin{lemma} \label{lemma:composition_continuity}
        Let $U \subset \R^n$ be a bounded Lipschitz domain, $0<s<1$, and $1\leq p<\infty$. Suppose that $f : \R^m \to \R$ is Lipschitz. Then, the composition map $u \mapsto f\circ u$ is continuous from $W^{s,p}(U;\R^m)$ into $W^{s,p}(U)$.
    \end{lemma}

    We also need the following.
    \begin{lemma} \label{lemma:comp_diffeo_wsp}
    Let $U, U' \subset \R^n$ be bounded Lipschitz domains in $\R^n$ and $\Psi:  U \to U'$ a $C^1$ diffeomorphism with inverse $\Phi$, such that $\Psi \in C^1(\overline U;\R^n)$ and $\Phi \in C^1(\overline U';\R^n)$. Let $0< s,\alpha < 1$, $1 \le p < \infty$. Then, $v \mapsto v \circ \Phi$ is a bounded linear map from $W^{s,p}(U)$ (respectively $c^{0,\alpha}(U)$ or $BV(U)$) into $W^{s,p}(U')$ (respectively $c^{0,\alpha}(U')$ or $BV(U')$). Furthermore, the following estimates hold with $L := \|D\Psi\|_{C^0(\overline U)}$ and $M := \|D\Phi\|_{C^0(\overline U')}$.
    \begin{align}
        [v\circ \Phi]_{W^{s,p}(U')} &\leq L^{\frac{2n}p} M^{\frac{n}{p}+s} [v]_{W^{s,p}(U)} \quad \forall v \in W^{s,p}(U), \label{eq:frac_Sob_diffeo}\\
        [v \circ \Phi]_{0,\alpha, U'| \epsilon} &\leq M^\alpha [v]_{0,\alpha,U|M\epsilon} \quad \forall v \in c^{0,\alpha}(U),\\
        \|v \circ \Phi\|_{BV(U')} &\leq C(n,U) L^{2n} \max \{M^{n+1}, 1\} \|v\|_{BV(U)} \label{eq:BV_est_for_v_circ_Phi} \quad \forall v \in BV(U).
    \end{align}
    \end{lemma}
\begin{proof}
    In the little Hölder setting, the result follows from Lemma \ref{lemma:c0alpha_equivalence}(iii). 
    
    If $v \in W^{s,p}(U)$, then the change variable formula gives
    \begin{align*}
        \|v \circ \Phi\|_{L^p(U')}^p &= \int_{U} |v(x)|^p |\det D\Psi(x)|dx \leq L^n \|v\|_{L^p(U)}^p,\\
        [v\circ\Phi]_{W^{s,p}(U')}^p &= \int_U\int_U \frac{|v(x)-v(y)|^p}{|\Psi(x)-\Psi(y)|^{n+sp}} |\det D\Psi(x)| |\det D\Psi(y)| dx dy\\
        &\leq L^{2n} M^{n+sp} [v]_{W^{s,p}(U)}^p
    \end{align*}
    where we also used the estimate $|\Psi(x)-\Psi(y)| \geq M^{-1} |x-y|$. This proves the lemma in the $W^{s,p}$ setting.
    
    Finally, for the $BV$ case we use Bourgain -- Brezis -- Mironescu \cite[Corollary 5]{bbm_01}. By the latter result, there exist $C_1, C_2 > 0$ depending on $n, U, U'$ such that
    \begin{align*}
        [v \circ \Phi]_{BV(U')} &\leq C_1 \liminf_{s \nearrow 1} (1-s) [v \circ \Phi]_{W^{s,1}(U')},\\
        [v]_{BV(U)} &\geq \frac{1}{C_2} \limsup_{s \nearrow 1} (1-s)[v]_{W^{s,1}(U)}.
    \end{align*}
    Estimate \eqref{eq:BV_est_for_v_circ_Phi} follows from the above and \eqref{eq:frac_Sob_diffeo}. The desired boundedness property in the $BV$ setting follows form this statement and the estimate $\|v \circ \Phi\|_{L^1(U')} \leq L^n \|v\|_{L^1(U)}$ proven previously.
\end{proof}
\begin{lemma} \label{lemma:inverse_wsp}
    Let $U, U', \Psi, \Phi, s, \alpha, p$ be as in Lemma \ref{lemma:comp_diffeo_wsp}. If $\Psi$ additionally belongs to $W^{1+s,p}(U; \R^n)$ (respectively $c^{1,\alpha}(U; \R^n)$ or $BV^2(U;\R^n)$, then $\Phi \in W^{1+s,p}(U'; \R^n)$ (respectively $c^{1,\alpha}(U';\R^n)$ or $BV^2(U';\R^n)$).
\end{lemma}
\begin{proof}
    By the chain rule, the matrix $D\Phi(x)$ is the inverse matrix of $D\Psi(\Phi(x))$. Therefore, noting that $D\Psi$ takes values into the compact set
    \begin{equation*}
        K := \{ A \in \R^{n\times n} : \det A \geq \|\det D\Phi\|_{C^0(\overline U')}^{-1} > 0, \ |A| \leq \|D\Psi\|_{C^0}\},
    \end{equation*}
    the conclusion follows from Lemmas \ref{lemma:composition_continuity} and \ref{lemma:comp_diffeo_wsp}.
\end{proof}

\subsection{Some operations on distributional differential forms} \label{sec:forms_ops}
    
In this subsection we discuss the exterior derivative, the wedge product, the interior product, and pullbacks of distributional differential forms. Let us start with the exterior derivative. Since a $k$-form $\alpha \in \mathcal{D}_k(U)$ on an open subset $U$ of $n$-manifold $(M,g)$ is an $(n-k)$-current, its exterior derivative $d\alpha \in \mathcal{D}_{k-1}(U)$ can be defined by
\begin{equation*}
    d\alpha [\phi] := (-1)^{n-k+1}\alpha [d\phi], \quad \phi \in \mathcal{D}^{k-1}(U).
\end{equation*}
Then, $d = (-1)^{n-k+1}\partial$ on $k$-currents, where $\partial$ is the boundary of the current. This is consistent with Stokes' theorem and the Leibniz rule $d(\alpha \wedge \phi) = d\alpha \wedge \phi + (-1)^{k} \alpha \wedge d\phi$ for a $k$-form $\alpha$. It is clear that if $\alpha \in W^{\pm s,p}(U;\Lambda^k T^* \R^n)$ then $d\alpha \in W^{\pm s-1,p}(U;\Lambda^{k+1} T^* \R^n)$.

 For the definition of the wedge product and interior product, we use the following facts. Recall the notation $CH^\frac12(U;E)$ and associated spaces (for a bundle $E$ over $M$) defined at \eqref{eq:def_CH12}, \eqref{eq:def_CH12_0} and directly thereafter.
    \begin{proposition} \label{prop:WeakProducts}
        Let $(M,g)$ be a compact $n$-dimensional Riemannian manifold and $U \subseteq M$ a smooth open subset. Assume that $1\leq k,l\leq n$ with $k+l \leq n$. Then the following hold.
        \begin{enumerate}
            \item[(i)] There exists a unique bounded bilinear map
            \begin{equation*}
                \mathcal{W} : CH^\frac12(U;\Lambda^kT^*M) \times H^{-\frac12}(U;\Lambda^l T^*M) \to CH^{-\frac12}(U;\Lambda^{k+l}T^*M)
            \end{equation*}
            such that for any $\alpha \in C^\infty(\overline U;\Lambda^k T^*M), \beta \in \mathcal{D}^l(U), \psi \in \mathcal{D}^{n-k-l}(U)$,
            \begin{equation*}
                \mathcal{W}(\alpha,\beta)[\psi] = \int_U \alpha \wedge \beta \wedge \psi.
            \end{equation*}
            \item[(ii)] There exists a unique bounded bilinear map
            \begin{equation*}
                \mathcal{I} : CH^\frac12(U;TM) \times H^{-\frac12}(U;\Lambda^k T^*M) \to CH^{-\frac12}(U;\Lambda^{k-1}T^*M)
            \end{equation*}
            such that for any $X \in C^\infty(\overline U; T\R^n)$, $\alpha \in \mathcal{D}^{k}(U)$, $\psi \in \mathcal{D}^{n-k+1}(U)$,
            \begin{equation*}
                \mathcal{I}(X,\alpha)[\psi] = \int_U (\iota_X \alpha) \wedge \psi.
            \end{equation*}
        \end{enumerate}
    \end{proposition}
    To lighten up the notation, we adopt the following convention:
    \begin{enumerate}[(i)]
        \item If $\alpha \in CH^{\frac12}(U;\Lambda^k T^*\R^n)$ and $\beta \in H^{-\frac12}(U;\Lambda^l T^*\R^n)$, we will write $\alpha \wedge \beta$ for $\mathcal{W}(\alpha, \beta)$. 
        \item If $\alpha \in H^{-\frac12}(U;\Lambda^l T^*\R^n)$ and $\beta \in CH^{\frac12}(U;\Lambda^k T^*\R^n)$, we will write $\alpha \wedge \beta$ for $-\mathcal{W}(\beta,\alpha)$.
        \item If $X \in CH^{\frac12}(U;T\R^n)$, $\alpha \in H^{-\frac12}(U;\Lambda^k T^*\R^n)$, we will write $\iota_X \alpha$ for $\mathcal{I}(X,\alpha)$.
    \end{enumerate}
    
    \begin{proof}
        By a standard argument using a partition of unity, it suffices to consider the case where $U$ is contained in a chart. Fix any $\alpha \in C^\infty(\overline U; \Lambda^kT^*M), \beta \in \mathcal{D}^l(U)$, and $\psi \in \mathcal{D}^{n-k-l}(U)$. Then,
        \begin{align*}
            \Big| \int_{U} \alpha \wedge \beta \wedge \psi\Big| \leq \|\beta\|_{H^{-\frac12}(U;\Lambda^l T^*M)} \|\alpha\wedge \psi\|_{H^\frac12(U;\Lambda^{n-l}T^*M)}.
        \end{align*}
        Since the space $CH(U)$ of functions on the chart $U$ is an algebra and the norms are defined using a fixed choice of local coordinates, there holds
        \begin{align*}
            \|\alpha \wedge \psi\|_{H^\frac12(U;\Lambda^{n-l}T^*M)} 
            &\leq C\|\alpha\|_{CH^\frac12(U;\Lambda^k T^*M)}\|\psi\|_{CH^\frac12(U;\Lambda^{n-k-l}T^*M)},
        \end{align*}
        where $C$ depends on $U$, $k$, $l$, and the choice of coordinates. This completes the proof for the wedge product; that for the interior product is analogous.
    \end{proof}

    Finally, we discuss pullback and its interaction with the interior product and exterior derivative. Let $U, U' \subset \R^n$ be smooth bounded domains and $\Psi : \overline U \to \overline U'$ a bijection in $C^1(\overline U;\R^n) \cap H^\frac32(U;\R^n)$ with inverse $\Phi \in C^1(\overline U';\R^2) \cap H^\frac32(U';\R^n)$. Let $1 \leq k \leq n$ and $\alpha \in H^{-\frac12}(U;\Lambda^k T^*\R^n)$. Then, we define the pullback $\Phi^*\alpha \in CH^{-\frac12}(U;\Lambda^k T^*\R^n)$ by
        \begin{equation*}
            (\Phi^*\alpha) [ \phi ] := \alpha [\Psi^* \phi], \quad \phi \in CH^\frac12_0(U'; \Lambda^{n-k}T^*\R^n).
    \end{equation*}

    \begin{lemma} \label{lemma:pullback_commutation}
        Let $U,U',\Psi, \Phi$ be as above with $n=2$. Let $X \in CH^\frac12(U;T\R^2)$ and $\alpha \in H^\frac12(U;T^*\R^2)$. Then, with products defined in the sense of Proposition \ref{prop:WeakProducts}, the following equality holds in $CH^{-\frac12}(U';\Lambda^2 T^* \R^2)$:
        \begin{equation*}
            \Phi^*(\iota_{X}d\alpha) = \iota_{(\Psi_*X)} d(\Phi^*\alpha).
        \end{equation*}
    \end{lemma}
    \begin{proof}
        Let $\phi \in \mathcal{D}^1(U')$ be an arbitrary test one-form on $U'$. Then, for any sequences $\{X_\epsilon\} \subset C^\infty(\overline U;T\R^2)$ with $X_\epsilon \to X$ in $CH^\frac12(U;T\R^2)$ and $\{\alpha_\epsilon\} \subset C^\infty(\overline U;T^*\R^2)$ with $\alpha_\epsilon \to \alpha$ in $H^\frac12(U;T^*\R^2)$, we have by Proposition \ref{prop:WeakProducts} that
        \begin{align} 
                (\Phi^*(\iota_X d\alpha))[\phi] &= \iota_X d\alpha[\Psi^*\phi]\nonumber\\
                    &= \lim_{\epsilon \searrow 0} \int_U \big( \iota_{X_\epsilon}d\alpha_\epsilon\big) \wedge \Psi^*\phi\nonumber\\
                    &= \lim_{\epsilon \searrow 0} \int_U \Psi^*\Big( \Phi^*\big( \iota_{X_\epsilon}d\alpha_\epsilon\big) \wedge \phi\Big).
                    \label{eq:pullback_before_exchange}
        \end{align}
        Since, $X_\epsilon$ and $\alpha_\epsilon$ are smooth, the equality
        \begin{equation*}
            \Phi^*(\iota_{X_\epsilon} \alpha_\epsilon) = \iota_{\Psi_*X_\epsilon} (\Phi^*d\alpha_\epsilon),
        \end{equation*}
        holds classically (both sides are at least continuous). At this point we would like to compute $\Phi^*(d\alpha_\epsilon)$. For this, we claim  that
        \begin{equation} \label{eq:pullback_d_commute}
            \Phi^*(d\beta) = d(\Phi^*\beta) \quad \text{for any } \beta \in C^\infty(U;T^*\R^2)
        \end{equation}
        and both sides\footnote{A priori we only know that $\Phi^*\beta$ and $\Phi^*(d\beta)$ are $CH^\frac12$.} of the equation therefore belong to $CH^\frac12(U';\Lambda^2 T^* \R^2)$. Supposing this claim, we continue the proof of the lemma. By \eqref{eq:pullback_d_commute} we have 
        $\Phi^* d\alpha_\epsilon = d(\Phi^*\alpha_\epsilon)$ and thus $\Phi^*(\iota_{X_\epsilon} \alpha_\epsilon) = \iota_{\Psi_*X_\epsilon} d(\Phi^*\alpha_\epsilon)$. Inserting this into \eqref{eq:pullback_before_exchange} we have
        \begin{align*}
            (\Phi^*(\iota_X d\alpha))[\phi] &= \lim_{\epsilon \searrow 0} \int_U \Psi^* \big( \iota_{\Psi_*X_\epsilon}d(\Phi^*\alpha_\epsilon) \wedge \phi\big)\\
            &= \lim_{\epsilon \searrow 0} \int_{U'} \iota_{\Psi_*X_\epsilon}d(\Phi^*\alpha_\epsilon) \wedge \phi.
        \end{align*}
        Lemma \ref{lemma:comp_diffeo_wsp} implies that $\Psi_*X_\epsilon \to \Psi_*X$ in $CH^\frac12(U;T\R^2)$ and $\Phi^*\alpha_\epsilon \to \Phi^*\alpha$ in $H^\frac12(U;T^*\R^2)$. It follows from Proposition \ref{prop:WeakProducts} that $\iota_{\Psi_*X_\epsilon} d(\Phi^*\alpha_\epsilon) \to \iota_{\Phi_*X} d(\Phi^*\alpha)$ in $CH^{-\frac12}(U';T^*\R^2)$. Therefore the lemma will be proved once the claim \eqref{eq:pullback_d_commute} is proved.
        
        To prove the claim \eqref{eq:pullback_d_commute}, we write $\beta = \beta_i dx^i$, where $\beta_i$ are smooth functions. Then
        \begin{equation*}
            \Phi^*\beta = \beta_i(\Phi)d\Phi^i,
        \end{equation*}
        and, considering any smooth sequence $\Phi_\epsilon \to \Phi$ in $C^1(\overline U';\R^2)$, we have in the sense of distributions that 
        \begin{equation*}
            d(\Phi^*\beta) = \lim_{\epsilon\searrow 0} d\Big(\beta_i(\Phi)d\Phi^i_\epsilon\Big) = \lim_{\epsilon\searrow 0} \partial_j \beta_i(\Phi) d\Phi^j \wedge d\Phi^i_{\epsilon} = \Phi^*(d\beta),
        \end{equation*}
        as claimed. This concludes the proof.
    \end{proof}

\subsection{Lifting of certain $\mathbb{S}^1$-valued maps} \label{sec:lifting}

We recall a simple existence result for liftings of $\mathbb{S}^1$ valued maps. For a comprehensive treatment of circle-valued maps and the lifting problem, see the monograph of Brezis and Mironescu \cite{brezis_mironescu_21}. Because we restrict ourselves to continuous maps, the existence and uniqueness for the lifting follows from more elementary arguments.
\begin{proposition} \label{prop:lifting}
Let $U \subset \R^n$ be a bounded Lipschitz simply connected domain, and $u \in C(\overline{U}; \mathbb{S}^1)$, and $\theta \in C(\overline{U};\R)$ a lifting\footnote{Such lifting exists uniquely modulo $2\pi \mathbb{Z}$; see, e.g., \cite[Lemma 1.1]{brezis_mironescu_21}.} such that $u = (\cos\theta, \sin\theta)$. If $u \in W^{s,p}(U; \mathbb{S}^1)$ with $s \in (0,1)$ and $p \in [1,\infty)$, then $\theta \in W^{s,p}(U)$.
\end{proposition}

\begin{proof}
    Cover $\mathbb{S}^1$ by two arcs $A_{\pm} := \{z \in \mathbb{S}^1 : |z\mp1|> \frac12\}$. Let $U_\pm := \{x \in U : u(x) \in A_\pm\}$. Then $\{U_-, U_+\}$ is an open cover for $U$, and
    \begin{equation*}
        \theta = \log_\pm(u) \quad \text{in } U_\pm
    \end{equation*}
    for some smooth branches $\log_\pm$ of the complex logarithm defined on $\mathbb{C}\setminus \R_\pm \supset \overline{A}_\pm$, where we have identified $\R^2 \simeq \mathbb{C}$. Since $u$ takes values there in $U_\pm$, we have $\theta \in W^{s,p}(U_\pm)$ by Lemma \ref{lemma:composition_continuity}. 
    
    We proceed to check that $\theta \in W^{s,p}(U)$, namely that $[u]_{W^{s,p}(U)} < \infty$. Split $U = U_1 \cup U_2 \cup U_3$ into three disjoint sets, where $U_1 := U_+ \setminus U_-$, $U_2 := U_+ \cap U_-$, and $U_3 := U_- \setminus U_+$. Accordingly, we split the integration over $U \times U$ in the definition of the $W^{s,p}(U)$ seminorm of $\theta$ into integration over sets of the form $U_i \times U_j$. All such integrals except for the ones over $U_1 \times U_3$ and $U_3 \times U_1$ are bounded by the seminorms $[\theta]_{W^{s,p}(U_+)}$ or $[\theta]_{W^{s,p}(U_-)}$, so it remains to bound the integral over $U_1 \times U_3$ (the other is the same). To this end it suffices to show that there exists $\delta > 0$ such that
    \begin{equation*}
    |x-y| \geq \delta \quad \forall x \in U_1, y\in U_3.
    \end{equation*}
    Observe that if $x \in U_1$, then $x \notin U_-$ and hence $|u(x)+1| \leq \frac12$. Likewise, if $y \in U_3$, then $|u(y) - 1| \leq \frac12$. It follows that $|u(x)-u(y)| \geq 1$ for all $x \in U_1$ and $y \in U_3$. The existence of $\delta$ follows from the uniform continuity of $u$ on $\overline U$.
\end{proof}

\section{Cartan's equations: Proof of Theorem \ref{thm:cartan_main}} \label{sec:proof_cartan}

We are now in a position to give the proof of Theorem \ref{thm:cartan_main}. We start with the following lemma relating a rough orthonormal frame to a smooth one via lifting.

\begin{lemma} \label{lemma:lifting_frame}
     Let $U \subset \Sigma$ be a smooth, simply connected chart on a smooth surface $(\Sigma, g)$, and $s \in (0,1)$, $p \in [1,\infty)$. Suppose that $\{\eta^i\}_{i=1}^2 \subset W^{s,p}(U; T^* \R^2) \cap C^0(\overline U; T^* \R^2)$ and $\{\hat\eta^i\}_{i=1}^2 \subset C^\infty(\overline U; T^*\R^2)$ are two orthonormal coframes on $(U,g)$ with the same orientation\footnote{Orientation is well-defined due to the continuity assumption.}. Then, there exists a lifting $\theta \in W^{s,p}(U) \cap C^0(\overline U)$, unique modulo $2\pi$, such that
     \begin{equation} \label{eq:lifting_rel}
        \begin{cases}
            \eta^1 &= \cos\theta ~\hat\eta^1 + \sin\theta~ \hat\eta^2\\
            \eta^2 &= -\sin\theta ~\hat\eta^1 + \cos\theta~ \hat\eta^2.
        \end{cases}
     \end{equation}
\end{lemma}
\begin{proof}
    In local coordinates we have
    \begin{equation*}
    \eta^i = \eta^i_j dx^j,\quad \hat\eta^i = \hat\eta^i_j dx^j
    \end{equation*}
    for $\eta^i_j \in W^{s,p}(U) \cap C^0(\overline U)$, $\hat\eta^i_j \in C^\infty(\overline U)$ with the orthonormality relation
    \begin{equation*}
        g^{kl} \eta^i_k \eta^j_l = \delta_{ij} = g^{kl} \hat\eta^i_k \hat\eta^j_l \quad \forall i, j \in \{1,2\}.
    \end{equation*}
    In particular, $\det(\hat{\eta}^i_j) = \det(\eta^i_j) = \sqrt{\det(g_{ij})}$. Since at each point in $\overline U$, $\hat\eta^1, \hat\eta^2$ form a basis for the cotangent space, we can write
    \begin{equation*}
        \eta^1 = u_1\hat\eta^1 + u_2\hat\eta^2.
    \end{equation*}
    The map $u : \overline U \to \R^2$ takes values in $\mathbb{S}^1$ as $|\eta^1|_g = 1$ and $\hat\eta^1,\hat\eta^2$ are orthonormal with respect to $g$. Moreover, $u$ is given by
    \begin{equation*}
        u = (\hat\eta^i_j)^{-1}(\eta^1_1,\eta^1_2)^T
    \end{equation*}
    which is well-defined since $\det(\hat\eta^i_j) = \sqrt{\det(g_{ij})} \geq \epsilon > 0$ on $U$, and hence
    \begin{equation*}
        u \in W^{s,p}(U;\mathbb{S}^1) \cap C^0(\overline U; \mathbb{S}^1).
    \end{equation*}
    The conclusion follows from Proposition \ref{prop:lifting}.
\end{proof}

Before giving the proof of Theorem \ref{thm:cartan_main} we recall our convention for the wedge product and interior product of rough differential forms following Proposition \ref{prop:WeakProducts}.
    
\begin{proof}[Proof of Theorem \ref{thm:cartan_main}]
    By Proposition \ref{prop:WeakProducts}, $\omega$ is globally well-defined via \eqref{eq:conn_form_formula} and belongs to $CH^{-\frac12}_{loc}(U;T^*M)$. Since distributions are local, it suffices to consider the case that $U$ is a smooth simply connected chart and $\{\eta^i\}_{i=1}^2 \subset CH^\frac12(U;T^*\R^2)$, where we need to show that $\omega$ belongs to $H^{-\frac12}(U;T^*\R^2)$, is the unique solution to \eqref{eq:cartan1} in $H^{-\frac{1}{2}}(U,T^*\R^2)$ in the sense of Proposition \ref{prop:WeakProducts}, and satisfies \eqref{eq:cartan2} in the sense of distributions on $U$.

    First, we show uniqueness in $H^{-\frac12}(U;T^*\R^2)$. Suppose that $\omega_1, \omega_2 \in H^{-\frac12}(U;T^*\R^2)$ are two solutions. Then $\beta := \omega_1 - \omega_2 \in H^{-\frac12}(U;T^*\R^2)$ satisfies
    \begin{equation*}
        0 = (\beta \wedge \eta^i)[\phi] = \beta[\phi\eta^i] \quad \forall \phi \in CH^\frac12_0(U), \quad i = 1,2.
    \end{equation*}
    Take $\psi \in \mathcal{D}^1(U)$. Then, $\psi = \sum_{i=1}^2 \phi_i \eta^i$, where $\phi_i := g(\eta^i, \psi) \in CH^\frac12_0(U)$. Hence, $\beta[\psi] = \sum_{i=1}^2 \beta[\phi_i \eta^i] = 0$. Since $\mathcal{D}^1(U)$ is dense in $H^{\frac12}_0(U;T^*\R^2)$ and $\beta$ belongs to $H^{-\frac12}(U;T^*\R^2)$, it follows that $\beta = 0$.
    
    Next, fix any smooth orthonormal coframe $\hat\eta^i \in C^\infty(\overline U; T^* \R^2)$, which exists as $g$ is smooth and $U$ is a chart. Since $U$ is simply connected, we can apply Lemma \ref{lemma:lifting_frame} to obtain a lifting $\theta \in C^0(\overline U) \cap H^\frac12(U)$ so that \eqref{eq:lifting_rel} holds. We claim that
    \begin{equation} \label{eq:omega_equals_lifting_exprn}
        \omega = \hat\omega - d\theta \quad \text{in } CH^{-\frac12}(U;T^*\R^2),
    \end{equation}
    where $\hat\omega$ is the connection form of the smooth frame. Once this is proved, it is immediate that $\omega \in H^{-\frac{1}{2}}(U,T^*\R^2)$. Moreover, since the equation $d\hat \omega = K_g dvol_g$ holds classically and the identity $d^2 = 0$ holds for distributions, the second structural equation \eqref{eq:cartan2} follows from \eqref{eq:omega_equals_lifting_exprn}.
    
    To prove the claim \eqref{eq:omega_equals_lifting_exprn}, we introduce the following regularising sequence. Take an extension $\theta \in (C_c\cap H^\frac12)(\R^2)$ and let $\theta_\epsilon := \theta * \rho_\epsilon$ be its convolution with a standard mollifier kernel. Define
    \begin{equation} \label{eq:eta_eps_def}
    \begin{split}
        {\eta}^1_{(\epsilon)} &:= \cos\theta_\epsilon\, \hat\eta^1 + \sin\theta_\epsilon\, \hat\eta^2,\\
        {\eta}^2_{(\epsilon)} &:= -\sin\theta_\epsilon\, \hat\eta^1 + \cos\theta_\epsilon\, \hat\eta^2.
    \end{split}
    \end{equation}
    Then, $\{\eta^i_{(\epsilon)}\}_{i=1}^2$ is a smooth orthonormal coframe on $(U,g)$. Moreover, a direct computation shows that its connection $1$-form is
    \begin{equation*}
        \omega_{(\epsilon)} = \hat\omega - d\theta_\epsilon,
    \end{equation*} 
    which satisfies Cartan's structural equations classically. By Lemma \ref{lemma:composition_continuity} with $f = \sin$ or $\cos$ and the convergence $\theta_\epsilon \to \theta$ in $C^0(\overline U) \cap H^\frac12(U)$, we have
    \begin{equation} \label{eq:approx_frame_cartan_conv}
        \eta^i_{(\epsilon)} \to \eta^i \text{ in } CH^\frac12(U;T^*\R^2).
    \end{equation}
    The dual frame $e_{i,(\epsilon)}$ likewise converges in $CH^\frac12(U;T\R^2)$ to $e_i$ as $\epsilon \searrow 0$. Hence, by Proposition \ref{prop:WeakProducts}, in the space $CH^{-\frac12}(U;T^*\R^2)$, there holds
    \begin{align*}
        \omega &= \lim_{\epsilon\searrow 0} \iota_{e_{2,(\epsilon)}}d\eta^1_{(\epsilon)} - \iota_{e_{1,(\epsilon)}}d\eta^2_{(\epsilon)}\\
        &= \lim_{\epsilon \searrow 0} (\hat\omega - d\theta_\epsilon)\\
        &= \hat\omega - d\theta.
    \end{align*}
  
    Finally, it remains to show that $\hat\omega - d\theta$ verifies Cartan's first structural equations. Recall that $d\eta^1_{(\epsilon)} = \eta^2_{(\epsilon)} \wedge \omega_{(\epsilon)} = \eta^2_{(\epsilon)} \wedge (\hat\omega - d\theta_\epsilon)$. We have, with limits in the sense of $CH^{-\frac12}(U;T^*\R^2)$, that:
    \begin{equation*} 
        \begin{split}
        d\eta^1  &= \lim_{\epsilon \searrow 0} d\eta^1_{(\epsilon)}\\
        &= \lim_{\epsilon \searrow 0} (\eta^2_{(\epsilon)} \wedge (\hat\omega - d\theta_\epsilon))\\
        &= \eta^2 \wedge (\hat\omega - d\theta)\\
        &= \eta^2 \wedge \omega,
        \end{split}
    \end{equation*}
    where we have used Proposition \ref{prop:WeakProducts} in the third equality (which applies since $\eta^2_{(\epsilon)} \to \eta^2$ in $CH^\frac12(U;T^*\R^2)$ and $\hat\omega-d\theta_{\epsilon} \to \hat\omega-d\theta$ in $H^{-\frac12}(U;T^*\R^2)$). This shows that $d\eta^1 = \eta^2 \wedge \omega$ as an equation in $CH^{-\frac12}(U;T^*\R^2)$ with the wedge product $\eta^2 \wedge \omega$ understood in the sense of Propsition \ref{prop:WeakProducts}. Likewise, $d\eta^2 = - \eta^1 \wedge \omega$. 
\end{proof}

\section{Gauss's equation: Proof of Theorems \ref{thm:Gauss_23} and \ref{thm:capacity}} \label{sec:gauss}

We now turn our attention to the Gauss equation. In \S \ref{sec:surfs_are_locally_graphical}, we consider the regularity of the local graphical representations of an embedded surface. In \S \ref{sec:proof_cartan_cor}, we prove Corollary \ref{cor:cartan_particular_frame} from Theorem \ref{thm:cartan_main}. Sections \ref{sec:W23_Gauss} and \ref{sec:BV_Gauss} are devoted to the proof of Theorem \ref{thm:Gauss_23} in cases (ii) and (i), respectively. In both cases, the result is deduced from Corollary \ref{cor:cartan_particular_frame} and an appropriate chain rule. Finally, in \S \ref{sec:capacity}, we prove Theorem \ref{thm:capacity}.

\subsection{The local parametrisation as a graph} \label{sec:surfs_are_locally_graphical}

Consider a $C^1$ isometric immersion $u : (U,g) \to \R^3$, where $g$ is a smooth Riemannian metric on a bounded Lipschitz domain $U \subset \R^2$. Pick some $x_0 \in U$, and composing with a rigid motion if necessary, we can assume that the unit normal at $x_0$ is $e_3=(0,0,1)$.
Decompose $u$ as
\begin{equation*}
u =: (\Psi,v)
\end{equation*}
where $\Psi = (u^1,u^2)$ is the in-plane component and $v=u^3$ the normal component. Then $D\Psi$ is invertible at $x_0$. To see this, note that as $u$ is an immersion, $Du$ has full rank, while at $x_0$
\begin{equation*}
    Dv(x_0) = e_3^T Du(x_0)  = 0.
\end{equation*}
Hence, by the inverse function theorem, and shrinking $U$ if necessary, we may assume that $\Psi : U \to U' \subset \R^2$ is a $C^1$ diffeomorphism onto its image. Furthermore, if $\Phi : U' \to U$ is the inverse of $\Psi$ and $f := v \circ \Phi$, then
\begin{equation*}
    u(U) = G_f(U'); \quad G_f(x) := (x,f(x)),
\end{equation*}
so that the embedded surface is locally a graph. As $u$ is isometric, we have
\begin{equation} \label{eq:gpr_from_f}
    g' := I+Df \otimes Df = G_f^*e = \Psi^*g
\end{equation}
where $e$ is the Euclidean metric on $\R^3$. We observe first that $f$ is as regular as $u$.
\begin{lemma} \label{lemma:f_smooth_as_u}
    Let $u : U \subset \R^2 \to \R^3$ satisfy $\partial_i u \cdot \partial_j u = g_{ij} \in C^\infty$. Assume that $u$ has regularity $C^1(\overline U; \R^3) \cap W^{1+s,p}(U;\R^3)$ or $c^{1,\alpha}\cap BV^2(U;\R^3)$. Parametrize the surface as a graph as above, so that $u(U) = G_f(U')$. Then, $f$ has regularity $C^1(\overline U') \cap W^{1+s,p}(U')$ (resp. $c^{1,\alpha} \cap BV^2(U')$). 
\end{lemma}
\begin{proof}
    It is clear that $f \in C^1(\overline U')$. Moreover, since $Df = Dv(\Phi)D\Phi$,  $Df$ belongs to $W^{s,p}(U';\R^2)$ (respectively $c^{0,\alpha} \cap BV(U';\R^2)$) by Lemmas \ref{lemma:composition_continuity}, \ref{lemma:comp_diffeo_wsp}, and \ref{lemma:inverse_wsp}.
\end{proof}

\subsection{Proof of Corollary \ref{cor:cartan_particular_frame}} \label{sec:proof_cartan_cor}
Let $u: U \to \R^3$ be an isometric embedding of $g$ of class $H^\frac32(U;\R^3) \cap C^1(\overline U; \R^3)$, and $U', \Psi, \Phi, v, f$ defined as in \S \ref{sec:surfs_are_locally_graphical}. Because the metric $g' := G_f^*e$ is not smooth, we cannot immediately apply Theorem \ref{thm:cartan_main}. We address this issue by pulling back onto $U$. 

\begin{proposition} \label{prop:pullback_Cartan}
    Let $\{\eta'^i\}_{i=1}^2 \in CH^\frac12(U';T^*\R^2)$ be any $g'$-orthonormal coframe\footnote{Such an orthonormal coframe exists, e.g. by applying the Gram-Schmidt process to a smooth coframe.}, with dual frame $\{e'_i\}_{i=1}^2$. Define a distributional $1$-form $\omega' \in CH^{-\frac12}(U';T^*\R^2)$ via Proposition \ref{prop:WeakProducts} by
    \begin{equation} \label{eq:def_omegapr}
        \omega' := \iota_{e'_2} d\eta'^1 - \iota_{e'_1} d\eta'^2,
    \end{equation}
    Then, there holds
    \begin{equation}
        d\omega' = K_g(\Phi)(1+|Df|^2)^\frac12 dx^1 \wedge dx^2 \quad \text{ in } \mathcal{D}'(U').
    \end{equation}
\end{proposition}

\begin{proof}
    Consider the pullback coframe $\eta^i := \Psi^*\eta'^i \in CH^\frac12(U;T^*\R^2)$, whose dual frame with respect to $g$ is $e_i := \Phi_*e'_i \in CH^\frac12(U;T\R^2)$. Since $\{\eta^i\}_{i=1}^2$ is an orthonormal coframe $(U,g)$ of class $CH^\frac12$, Theorem \ref{thm:cartan_main} implies that the distributional $1$-form
    \begin{equation*}
        \omega := \iota_{e_2}d\eta^1 - \iota_{e_1}d\eta^2
    \end{equation*}
    (defined via Proposition \ref{prop:WeakProducts}) belongs to $H^{-\frac12}(U;T^*\R^2)$ and satisfies Cartan's structural equations \eqref{eq:cartan1} and \eqref{eq:cartan2} in the sense of distributions. Moreover, by Lemma \ref{lemma:pullback_commutation}, $\omega'$ and $\omega$ are related via pullback:
    \begin{equation*}
        \omega' = \Phi^*\omega \quad \text{in } CH^{-\frac12}(U';T\R^2).
    \end{equation*}
    We compute for $\phi \in C^\infty_c(U')$:
    \begin{align*}
        d\omega'[\phi] 
            &= \omega'[d\phi] 
                =  (\Phi^*\omega)[d\phi]\\
            &= \omega[\Psi^*(d\phi)] 
                = \omega[d(\phi\circ\Psi)] 
                =  d\omega[\phi\circ\Psi]
                = \int_U \phi(\Psi)K_g dvol_g.
    \end{align*}
    where we have used Cartan's second structural equation $d\omega = K_g dvol_g$ for the last equality. Using the change of variable formula and noting that $\Phi^*dvol_g = dvol_{g'} = (1+|Df|^2)^\frac12 dx^1 \wedge dx^2$ (since $\Phi$ is an isometry), we therefore have
     \begin{align*}
        d\omega'[\phi] 
            &= \int_{U'} \phi K_g(\Phi) (1+|Df|^2)^\frac12 dx^1 \wedge dx^2 \text{ for all } \phi \in C_c^\infty(U').
    \end{align*}
    The proof is complete.
\end{proof}

\begin{proof}[{Proof of Corollary \ref{cor:cartan_particular_frame}}]
    We now make an explicit choice of coframe on $(U',g')$, by applying the Gram-Schmidt process to the coordinate coframe $\{dx^i\}_{i=1}^2$, namely
    \begin{equation} \label{eq:coframe_choice}
        \begin{split}
            \eta'^1 &:= (1+|Df|^2)^\frac12 (1+f_{x_2}^2)^{-\frac12} dx^1,\\
            \eta'^2 &:= f_{x_1}f_{x_2} (1+f_{x_2}^2)^{-\frac12} dx^1 + (1+f_{x_2}^2)^{\frac12} dx^2.
        \end{split}
    \end{equation}
    It is clear from Lemma \ref{lemma:composition_continuity} that $\{\eta'^i\}_{i=1}^2 \subset CH^\frac12(U';T^*\R^2)$. Let $\omega'$ be defined from $\{\eta'^i\}$ by \eqref{eq:def_omegapr} via Proposition \ref{prop:WeakProducts}. By Proposition \ref{prop:pullback_Cartan} we only need to check that
    \begin{equation} \label{eq:omega_as_DetDw}
        d\omega' = \Det Dw \quad \text{in } {\mathcal{D}'(U')},
    \end{equation}
    where we recall $w = F(Df)$ and $F(z) = ((1+|z|^2)^{-\frac12} (1+z_2^2)^{-1} z_1, z_2)$. 
    
    Take a sequence $\{f_\epsilon\} \subset C^\infty(\overline U')$ such that $Df_\epsilon \to Df$ in $CH^\frac12(U';\R^2)$ as $\epsilon \searrow 0$. Consider the regularised metric $g'_{(\epsilon)} := 1+Df_{\epsilon} \otimes Df_{\epsilon}$ on $U'$, and define the regularising $g_{(\epsilon)}'$-orthonormal coframe $\eta'^i_{(\epsilon)}$ by replacing $f$ by $f_\epsilon$ everywhere, i.e.,
    \begin{equation*}
        \begin{split}
            \eta'^1_{(\epsilon)} &:= (1+|Df_\epsilon|^2)^\frac12 (1+f_{\epsilon, x_2}^2)^{-\frac12} dx^1,\\
            \eta'^2_{(\epsilon)} &:= f_{\epsilon, x_1}f_{\epsilon, x_2} (1+f_{\epsilon, x_2}^2)^{-\frac12} dx^1 + (1+f_{\epsilon, x_2}^2)^{\frac12} dx^2.
        \end{split}
    \end{equation*}
    By Lemma \ref{lemma:composition_continuity} and the uniform convergence $Df_\epsilon \to Df$ in $\overline U'$, we have $\eta'^i_{(\epsilon)} \to \eta'^i$ in $CH^\frac12(U';T^*\R^2)$. 
    
    Let $\{e'_i\}_{i=1}^2$ be the $g'$-orthonormal frame dual to $\{\eta'^i\}_{i=1}^2$, namely
    \begin{equation} \label{eq:frame_choice}
        \begin{split}
            e'_1 &:= (1+|Df|^2)^{-\frac12}\big( (1+f_{x_2}^2)^\frac12 \partial_{x_1} - f_{x_1}f_{x_2}(1+f_{x_2}^2)^{-\frac12}\partial_{x_2} \big),\\
            e'_2 &:= (1+f_{x_2}^2)^{-\frac12} \partial_{x_2}.
        \end{split}
    \end{equation}
    Likewise, let $\{e'_{i, (\epsilon)}\}_{i=1}^2$ be the frame dual to $\{\eta'^i_{(\epsilon)}\}_{i=1}^2$ (which is $g'_{(\epsilon)}$-orthonormal and satisfies $\eta'^i_{(\epsilon)}(e'_{j,(\epsilon)}) = \delta_{ij}$), given by replacing $f$ by $f_\epsilon$ everywhere in \eqref{eq:frame_choice}. Similarly, $e'_{i,(\epsilon)} \to e'_i$ in $CH^\frac12(U';T\R^2)$.
    
    A direct computation shows that the connection form of $\{\eta'^i_{(\epsilon)}\}_{i=1}^2$ on $(U',g'_{(\epsilon)})$ is
    \begin{align*}
        \omega'_{(\epsilon)} &= \iota_{e'_{2,(\epsilon)}} d\eta'^1_{(\epsilon)} - \iota_{e'_{1,(\epsilon)}} d\eta'^2_{(\epsilon)}\\
        &= (1+|Df_\epsilon|^2)^{-\frac12}(1+f_{\epsilon, x_2}^2)^{-1} f_{\epsilon, x_1} df_{\epsilon, x_2},
    \end{align*}
    and hence that
    \begin{equation*}
        d\omega'_{(\epsilon)} = \det D(F(Df_\epsilon)).
    \end{equation*}
    Now, on one hand, by the convergence of $\eta'^i_{(\epsilon)}$ and $e'^i_{(\epsilon)}$ and Proposition \ref{prop:WeakProducts}, we have
    \begin{equation*}
        \omega'_{(\epsilon)} \to \omega' \quad \text{in } CH^{-\frac12}(U';T^*\R^2).
    \end{equation*}
    On the other hand, we have $F(Df_\epsilon) \to F(Df) = w$ in $H^\frac12(U';\R^2)$ by Lemma \ref{lemma:composition_continuity}. Appealing to \cite[Theorem 3]{brezis_nguyen_11}, it follows that, in the sense of distributions on $U'$,
    \begin{equation*}
        \lim_{\epsilon\searrow 0}d\omega'_{(\epsilon)} = \lim_{\epsilon\searrow 0}\det D(F(Df_\epsilon)) = \Det Dw,
    \end{equation*}
    and hence \eqref{eq:omega_as_DetDw} is proved.
\end{proof}

\subsection{The Gauss equation in $C^1 \cap W^{1+\frac23,3}$} \label{sec:W23_Gauss}

Next, we use Corollary \ref{cor:cartan_particular_frame} to prove the Gauss equation \eqref{eq:gauss1} for the case that $f \in C^1(\overline U') \cap W^{1+\frac23,3}(U')$. 

\begin{proof}[Proof of Theorem \ref{thm:Gauss_23}, case (ii)]
   By localisation of distributions, we can assume that $u \in C^1(\overline U; \R^3) \cap W^{1+\frac23,3}(U;\R^3)$. Then, by Lemma \ref{lemma:f_smooth_as_u} we have $f \in C^1(\overline U') \cap W^{1+\frac23,3}(U')$. Note that $W^{1+\frac23,3}(U;\R^3)$ embeds into $H^\frac32(U;\R^3)$, so Corollary \ref{cor:cartan_particular_frame} implies that $f \in C^1(\overline U') \cap W^{1+\frac23,3}(U')$ and
   \begin{equation*}
       \Det Dw = K_g(\Phi)(1+|Df|^2)^2 \quad \text{in } \mathcal{D}'(U'),
   \end{equation*}
   where $w = F(Df)$ and $F(z) = ((1+|z|^2)^{-\frac12}(1+z_2^2)^{-1}z_1,z_2)$. 
   
   We now wish to use a version of the chain rule to conclude. As pointed out earlier in the introduction, in the current case, this can be done in a straightforward manner with a regularisation procedure; the situation of case \eqref{case:GaussBV} is more delicate and will be treated in the next subsection. Indeed, take a regularising sequence $\{f_\epsilon\} \subset C^\infty(\overline U')$ such that $f_\epsilon \to f$ in $C^1(\overline U') \cap W^{1+\frac23,3}(U')$. Then, $F(Df_\epsilon) \to F(Df)$ in $W^{\frac23,3}(U';\R^2)$ by Lemma \ref{lemma:composition_continuity}. Recall that $v \mapsto \Det Dv$ is a well-defined continuous map from $W^{\frac23,3}(U';\R^2)$ into $(W^{\frac23,3}_0(U'))^*$ (see, e.g., \cite[Proposition 1]{gladbach_olbermann}). Therefore,
   \begin{equation*}
    \det D(F(Df_\epsilon)) \to \Det Dw \quad \text{in } (W^{\frac23,3}_0(U'))^*.
   \end{equation*}
   Moreover, by the chain rule for Jacobian determinant we have
   \begin{equation*}
       \det D(F(Df_\epsilon)) = \det DF(Df_\epsilon)\det D^2 f_\epsilon =(1+|Df_\epsilon|^2)^{-\frac32} \det D^2 f_\epsilon.
   \end{equation*}
   Lemma \ref{lemma:composition_continuity} implies that
   \begin{equation*}
       (1+|Df_\epsilon|^2)^\frac32 \to (1+|Df|^2)^\frac32 \quad \text{in } W^{\frac23,3}(U').
   \end{equation*}
   Therefore, for any $\psi \in C^\infty_c(U')$,
   \begin{align*}
       \Det D^2 f[\psi] &= \lim_{\epsilon \searrow 0} \int_{U'} \det D^2 f_\epsilon~ \psi ~dx\\
       &= \lim_{\epsilon \searrow 0} \Big\langle \det D(F(Df_\epsilon)), (1+|Df_\epsilon|^2)^\frac32 \psi \Big\rangle\\
       &= \Big\langle\Det Dw , (1+|Df|^2)^\frac32 \psi\Big\rangle\\
       &= \int_{U'} K_g(\Phi)(1+|Df|^2)^2 \psi~dx,
   \end{align*}
   where $\langle \cdot, \cdot\rangle$ denotes the pairing between $(W^{\frac23,3}_0(U'))^*$ and $W^{\frac23,3}_0(U')$.
\end{proof}

\subsection{The Gauss equation in $c^{1,\frac12} \cap BV^2$} \label{sec:BV_Gauss}

We next prove the Gauss equation for $f \in c^{1,\frac12} \cap BV^2$ (i.e. case (i) of Theorem \ref{thm:Gauss_23}). As in case (ii), we shall deduce the theorem from Corollary \ref{cor:cartan_particular_frame}, which says that
\begin{equation*}
    \Det Dw = K_g(\Phi)(1+|Df|^2)^\frac12 \quad \text{in } \mathcal{D}'(U')
\end{equation*}
where $w := F(Df)$ and $F : \R^2 \to \R^2$ is given by $F(z_1,z_2) = ((1+z_2^2)^{-1}(1+z_1^2+z_2^2)^{-\frac12}z_1,z_2)$. A direct calculation shows that $\det DF(z) = (1+|z|^2)^{-\frac32}$. Hence, the problem of recovering $\Det D^2f$ from $\Det Dw$ essentially rests on the validity of an appropriate chain rule $\Det D^2w = \Det DF(w) \Det Dw$ for the Jacobian determinant. Lacking a sufficient understanding of which space the distributional Jacobian maps $c^{0,\frac12} \cap BV(U')$ into, we prove the following  Proposition \ref{prop:BV_det_chain} by more ad hoc arguments. It is the key tool in the derivation of Theorem \ref{thm:Gauss_23}, case (i), from Corollary \ref{cor:cartan_particular_frame}.

\begin{proposition} \label{prop:BV_det_chain}
    Let $\Omega \subset \R^2$ be an open set, and $w \in c^{0,\frac12}_{loc}\cap BV_{loc}(\Omega;\R^2)$. Assume that $\Det Dw$ belongs to $L^1_{loc}(\Omega)$. Let $G : \R^2 \to \R^2$ be a $C^3$ function with bounded derivatives up to third order, and define $v := G(w)$, which belongs to $c^{0,\frac12}_{loc}\cap BV_{loc}(\Omega;\R^2)$ by Lemma \ref{lemma:composition_continuity}. Then,
    \begin{equation*}
        \Det Dv = \det DG(w) \Det Dw \quad \text{in } \mathcal{D}'(\Omega).
    \end{equation*}
\end{proposition}
Once Proposition \ref{prop:BV_det_chain} is proved, Theorem \ref{thm:Gauss_23}, case (i) can be deduced as follows.
\begin{proof}[Proof of Theorem \ref{thm:Gauss_23}, case (i)]
    As in case (ii), we deduce the theorem from Corollary \ref{cor:cartan_particular_frame}, which applies since by interpolation, $c^{1,\frac12} \cap BV^2(U;\R^3)$ embeds into $C^1(\overline U;\R^3) \cap H^\frac32(U;\R^3)$. Furthermore, by Lemma \ref{lemma:f_smooth_as_u}, the graphical representation $f : U' \to \R$ belongs to $c^{1,\frac12} \cap BV^2(U')$. Let $F$ be as in Corollary \ref{cor:cartan_particular_frame}. A straightforward computation shows that the range of $F$, as a map from $\R^2$ into $\R^2$, is $S := \{(w_1,w_2) \in \R^2 : |w_1| < (1+w_2^2)^{-1}\}$, and that the inverse of $F$ on $S$ is $G : S \to \R^2$ given by
    \begin{equation*}
        G(w_1,w_2) := \big( w_1 (1+w_2^2)^\frac32 (1-w_1^2(1+w_2^2)^2)^{-\frac12}, w_2\big), \quad w \in S.
    \end{equation*}
    Since $Df$ is bounded, $w := F(Df)$ takes values in a compact set $K \subset S$.  Hence, we may apply Lemma \ref{prop:BV_det_chain} to conclude that
    \begin{equation*}
        \Det D^2f = \Det D(G(w)) = \det DG(w) \Det Dw.
    \end{equation*}
    Since $\det DG(F(z)) = \det(DF(z))^{-1} = (1+|z|^2)^\frac32$ and $w = Df$, the proof is complete.
\end{proof}
It remains to prove the chain rule Proposition \ref{prop:BV_det_chain} for the Jacobian determinant, for which the following cancellation estimate is the focal point. (Compare \cite[Lemma 5.8]{giron_21}.) In contrast to the proof of case (ii) of Theorem \ref{thm:Gauss_23} where any smooth approximation works, we confine ourselves here to a particular smooth approximation of $v$ (via mollification) in order to obtain the estimate.
\begin{lemma} \label{lemma:cancellation}
Let $\Omega \subset \R^2$ be a bounded domain and $\alpha \in (0,1]$. Suppose that $v \in C^{0,\alpha} \cap BV(\Omega)$ and $h \in C_c^{0,\alpha}(\Omega)$. Let $\rho \in C^\infty_c(B_1(0))$ be a standard mollifier kernel and consider the mollification
\begin{equation*}
  v_\epsilon(x) := \int_{B_\epsilon(x)} v(y) \rho_\epsilon(x-y) dy, \quad x \in \Omega^\epsilon := \{z \in \Omega : \dist(z,\partial \Omega) > \epsilon\}.
\end{equation*}
Write $\mu := Dv \in \mathcal{M}(\Omega)^2$ for the distributional derivative of $v$, which is an $\R^2$-valued Radon measure on $\Omega$. Then, for any $\epsilon < \dist(\operatorname{supp} h,\partial\Omega)$, we have
\begin{equation} \label{eq:wedge_cancellation}
    \Big| \int_\Omega h D^\perp v_\epsilon \cdot d\mu \Big| \leq C \epsilon^{2\alpha-1} [h]_{0, \alpha; \Omega | \epsilon} [v]_{0,\alpha; \Omega | \epsilon}|\mu|(\Omega),
\end{equation}
where $C$ depends only on $\rho$.
\end{lemma}

\begin{proof}
    Let $\epsilon < \dist(\operatorname{supp} h,\partial\Omega)$. We will prove estimate \eqref{eq:wedge_cancellation} using the following four simple observations:
    \begin{enumerate}[(1)]
        \item $\int_{\R^2} D_x\rho_\epsilon(x-y)dy = 0$ for any $x \in \R^2$. 
        \item $\int_{\Omega} D_x^\perp \rho_\epsilon (x-y) \cdot d\mu(x) = 0$ for any $y \in \Omega^\epsilon$.
        \item $\int_\Omega D_x^\perp \rho_\epsilon(x-y) \cdot (v(x)d\mu(x)) = 0$ for any $y \in \Omega^\epsilon$.
        \item $\|D\rho_\epsilon\|_{L^1(\R^2)} = C \epsilon^{-1}$ where $C$ depends only on $\rho$.
    \end{enumerate}
    Points (1) and (4) are clear. Point (2) follows from the compact support of $\rho_\epsilon$ in $B_\epsilon$, the definition of the distributional derivative $\mu = Dv$, and the identity $\mathrm{div} \circ D^\perp = 0$. Point (3) uses the above together with the fact that, since $v \in BV(\Omega) \cap  C^0(\overline\Omega)$, we have $v^2 \in BV(\Omega)$ and $D(v^2) = 2v\mu$ by the chain rule in $BV$ (c.f. \cite[Theorem 3.96]{ambrosio_00}). By observation (1) we compute
    \begin{align*}
        \int_\Omega h(x)D^\perp v_\epsilon(x) \cdot d\mu(x) &= \int_\Omega \int_\Omega (h(x)-h(y))(v(y)-v(x)) D_x^\perp \rho_\epsilon(x-y) dy \cdot d\mu(x)\\
        &\quad + \int_\Omega \int_\Omega h(y)v(y) D_x^\perp \rho_\epsilon(x-y)dy \cdot d\mu(x)\\
        &\quad - \int_\Omega \int_\Omega h(y) D_x^\perp \rho_\epsilon(x-y)dy \cdot (v(x)d\mu(x))\\
        &=: I_1 + I_2 - I_3.
    \end{align*}
    Note that the integrands in $I_1, I_2, I_3$ are absolutely integrable over $\Omega \times \Omega$ with respect to the product measure $d|\mu|(x) \otimes dy$, so Fubini's theorem applies. Hence, $I_2 = I_3 = 0$ by observations (2) and (3), respectively. As the integrand is in $I_1$ supported inside the region $\{(x,y) \in \Omega \times \Omega : |x-y| \leq \epsilon\}$, we have
    \begin{align*}
        |I_1| &\leq [h]_{0,\alpha;\Omega|\epsilon}[v]_{0,\alpha;\Omega|\epsilon}  \epsilon^{2\alpha} \|D\rho_\epsilon\|_{L^1(\R^2)} |\mu|(\Omega)\\
        &\leq C [h]_{0,\alpha;\Omega|\epsilon} [v]_{0,\alpha;\Omega|\epsilon} |\mu|(\Omega) ~ \epsilon^{2\alpha-1}
    \end{align*}
    where (4) was used in the last line. The conclusion follows.
\end{proof}

In the proof of Proposition \ref{prop:BV_det_chain}, we will frequently make use of the following immediate consequence of Lemma \ref{lemma:cancellation}.

\begin{corollary}\label{Cor:WedCan}
Let $\Omega \subset \R^2$ be a bounded domain and $\alpha \in [\frac{1}{2},1]$. Suppose that $v \in c^{0,\alpha} \cap BV(\Omega)$ and let $v_\epsilon$ and $\mu$ be defined as in Lemma \ref{lemma:cancellation}. For any bounded family $\{h_\epsilon\} \subset  C^{0,\alpha}_c(\omega)$ with $\omega$ compactly contained in $\Omega$, it holds that
\begin{equation*} 
    \lim_{\epsilon \searrow 0} \int_\Omega h_\epsilon D^\perp v_\epsilon \cdot d\mu  = 0.
\end{equation*}
\end{corollary}

We also record the following simple observation regarding the explicit expression for $\Det Dw$ when $w$ belongs to $BV_{loc} \cap H^\frac12_{loc} \cap C^0(\Omega;\R^2)$.
\begin{lemma} \label{lemma:DetFormulaBV}
    Let $\Omega \subset \R^2$ be an open set and $w \in BV_{loc} \cap H^\frac12_{loc} \cap C^0(\Omega;\R^2)$. Let $\mu^i := Dw^i$, $i=1,2$, be the distributional derivative of $w^i$ (an $\R^2$-valued Radon measure in $\Omega$). Then, 
    \begin{equation*}
        \Det Dw [\psi] = -\int_\Omega w^1 D^\perp \psi \cdot d\mu^2 = \int_\Omega w^2 D^\perp \psi \cdot d\mu^1 \quad \forall \psi \in C^\infty_c(\Omega).
    \end{equation*}
\end{lemma}

\begin{proof} 
    Recall that the distributional Jacobian is a continuous bilinear operator on $H^\frac12_{loc}(\Omega)$ \cite[Theorem 3]{brezis_nguyen_11}. Let $\{w_\epsilon\}_{\epsilon \in (0,1)} \subset C^\infty(\Omega;\R^2)$ be such that $w_\epsilon \to w$ locally uniformly as well as in $H^\frac12_{loc}(\Omega,\R^2)$ as $\epsilon \rightarrow 0$. For fixed $\epsilon$, we have
    \begin{align*}
        \Det D(w^1_\epsilon, w^2)[\psi] &=  \lim_{\epsilon' \searrow 0} \int_\Omega \psi D^\perp w^1_\epsilon \cdot Dw^2_{\epsilon'} dx\\
        &= -\lim_{\epsilon' \searrow 0} \int_\Omega w^2_{\epsilon'} \mathrm{div}(\psi D^\perp w^1_{\epsilon}) dx\\
        &= \int_\Omega w^2 \mathrm{div}(w^1_\epsilon D^\perp \psi) dx\\
        &= -\int_\Omega w^1_\epsilon D^\perp \psi \cdot d\mu^2,
    \end{align*}
    where we integrated by parts in the second identity and we used the local uniform convergence $w^2_{\epsilon'} \to w$ in the third identity and the definition of $\mu^2 = Dw^2$ in the fourth identity. It follows that
    \begin{align*}
        \Det Dw[\psi] = \lim_{\epsilon \searrow 0} \Det D(w^1_\epsilon, w^2)[\psi] = -\lim_{\epsilon\searrow 0} \int_\Omega w^1_\epsilon D^\perp \psi \cdot d\mu^2 = -\int_\Omega w^1 D^\perp \psi \cdot d\mu^2.
    \end{align*}
    The other expression for $\Det Dw$ follows by reversing the roles of $w^1$ and $w^2$.
\end{proof}

Finally, we come to the proof of Proposition \ref{prop:BV_det_chain}.
\begin{proof}[Proof of Proposition \ref{prop:BV_det_chain}] 
    Fix a smaller open set $\omega \subset\subset \Omega$ and any test function $\psi \in C^\infty_c(\Omega)$ supported inside $\omega$. Let $\mu^i := Dw^i \in \mathcal{M}(\omega)^2$. By the chain rule in $BV$ (see, e.g., \cite[Theorem 3.96]{ambrosio_00}) and the continuity of $w$, we have $Dv^i = G^i_{z_j}(w)\mu^j$, and hence Lemma \ref{lemma:DetFormulaBV} yields
    \begin{equation*}
        \Det Dv[\psi] = -\int_\Omega v^1 D^\perp\psi \cdot d(Dv^2) = -\int_\Omega G^1(w) G^2_{z_j}(w) D^\perp \psi \cdot d\mu^j.
    \end{equation*}
    Letting $\Phi^j(z_1,z_2) := G^1(z)G^2_{z_j}(z)$ for $j=1,2$, and applying Lemma \ref{lemma:DetFormulaBV} once again on the right hand side, we obtain
    \begin{equation} \label{eq:sum_of_two_single_transforms}
        \Det Dv = \sum_{j=1}^2 \Det D(\Phi^j(w),w^j).
    \end{equation}
    
    Let us now compute $\Det D(\Phi^1(w),w^1)$. We shall argue by approximation by smooth functions. As usual, let $\rho \in C^\infty_c(B)$, $\rho_\epsilon(x):=\epsilon^{-2}\rho(x/\epsilon)$ be a standard mollifier and $w_\epsilon := w*\rho_\epsilon$. Then, $w_\epsilon \to w$ uniformly and in $H^\frac12(\omega;\R^2)$. By Lemma \ref{lemma:composition_continuity}, it follows that $\Phi^1(w_\epsilon) \to \Phi^1(w)$ uniformly and in $H^\frac12(\omega)$. Hence, by the $H^\frac12$-continuity of the Jacobian determinant,
    \begin{align*}
        \Det D(\Phi^1(w), w^1)[\psi] &= \lim_{\epsilon \searrow 0} \Det D(\Phi^1(w_\epsilon),w^1)[\psi]\\
        &= \lim_{\epsilon \searrow 0} \int_\Omega w^1 D^\perp \psi \cdot D(\Phi^1(w_\epsilon)) dx,
    \end{align*}
    where, for the last identiy, one can argue using Lemma \ref{lemma:DetFormulaBV} or a standard approximation argument.
    Using the definition of $\mu^1 = Dw^1$ we have
    \begin{align*}
        \int_\Omega w^1 D^\perp \psi \cdot D(\Phi^1(w_\epsilon)) dx &= -\int_\Omega w^1 \mathrm{div} \big(\psi D^\perp (\Phi^1(w_\epsilon))\big) dx\\
        &=\int_\Omega \psi D^\perp (\Phi^1(w_\epsilon)) \cdot d\mu^1\\
        &= \int_\Omega \psi ~ \Phi^1_{z_1}(w_\epsilon) D^\perp w^1_\epsilon \cdot d\mu^1 + \int_\Omega \psi ~ \Phi^1_{z_2}(w_\epsilon) D^\perp w^2_\epsilon \cdot d\mu^1\\
        &=: I_1(\epsilon) + I_2(\epsilon).
    \end{align*}
    Corollary \ref{Cor:WedCan} with $h_\epsilon := \psi \Phi^1_{z_1}(w_\epsilon)$ and $v := w^1$ implies that
    \begin{equation*}
        \lim_{\epsilon \searrow 0} I_1(\epsilon) = 0.
    \end{equation*}
    As for $I_2$, we compute
    \begin{align}
        I_2(\epsilon) &= \int_\Omega \Big( D^\perp \big(\psi \Phi^1_{z_2}(w_\epsilon) w^2_\epsilon\big) - w^2_\epsilon D^\perp (\psi \Phi^1_{z_2}(w_\epsilon)\big)\Big) \cdot d\mu^1 \nonumber\\
        &= -\int_\Omega w^2_\epsilon D^\perp\big(\psi \Phi^1_{z_2}(w_\epsilon)\big) \cdot d\mu^1.\label{eq:I2_w_replaces_weps}
    \end{align}
    where we have used the definition of $\mu^1 = Dw^1$ and the identity $\mathrm{div} \circ D^\perp = 0$. Since $w \in c^{0,\frac{1}{2}}_{{loc}}(\Omega,\R^2)$, we have by Lemmas \ref{lemma:c0alpha_convolution} and \ref{lemma:c0alpha_equivalence}(iii) that
    \begin{equation*}
        \|(w^2 - w^2_\epsilon)D^\perp \big(\psi \Phi^1_{z_2}(w_\epsilon)\big)\|_{C^0(\overline\omega)} \leq [w-w_\epsilon]_{0,\frac12;V|\epsilon} \epsilon^\frac12 [\psi \Phi^1_{z_2} (w_\epsilon)]_{0,\frac12;V|\epsilon} \epsilon^{-\frac12} \stackrel{\epsilon \rightarrow 0}{\longrightarrow} 0,
    \end{equation*}
    where $V$ is any open set with $\omega \subset\subset V \subset\subset \Omega$. Returning to \eqref{eq:I2_w_replaces_weps}, using Lemma \ref{lemma:DetFormulaBV} and recalling that $\Det Dw \in L^1_{loc}(\Omega)$, we therefore have
    \begin{align*}
        \lim_{\epsilon \searrow 0} I_2(\epsilon)
            & = - \lim_{\epsilon \rightarrow 0} \int_\Omega w^2 D^\perp\big(\psi \Phi^1_{z_2}(w_\epsilon)\big) \cdot d\mu^1\\
            &= - \lim_{\epsilon \searrow 0} \Det Dw [\psi \Phi^1_{z_2}(w_\epsilon)] 
                = - \lim_{\epsilon \searrow 0} \int_\Omega \psi \Phi^1_{z_2}(w_\epsilon)\, \Det Dw\,dx\\
            &= -\int_\Omega \psi \Phi^1_{z_2}(w)\,\Det Dw\, dx.
    \end{align*}
    Summarizing, we have thus shown that
    \begin{equation*}
        \Det D(\Phi^1(w),w^1)[\psi] =  -\int_\Omega \psi \Phi^1_{z_2}(w)\,\Det Dw dx,
    \end{equation*}
    i.e.
    \[
    \Det D(\Phi^1(w),w^1) = -\Phi^1_{z_2}(w)\,\Det Dw.
    \]
    By a similar calculation,
    \begin{equation*}
        \Det D(\Phi^2(w), w^2) = \Phi^2_{z_1} \Det Dw.
    \end{equation*}
    Recalling \eqref{eq:sum_of_two_single_transforms}, we thus have
    \begin{align*}
        \Det Dv &= (\Phi^2_{z_1}-\Phi^1_{z_2})(w) \Det Dw.
    \end{align*}
    Noting that $\Phi^2_{z_1} - \Phi^1_{z_2} = (G^1G^2_{z_2})_{z_1} - (G^1 G^2_{z_1})_{z_2} = \det DG$, we arrive at the desired conclusion.
\end{proof}

\subsection{Removable singularities for the Gauss equation} \label{sec:capacity}

We now ask whether an isometric embedding which is locally $W^{1+\frac23,3}$ away from a small `bad set' $S$ and globally $W^{1+s,p}$ must satisfy the Gauss equation for some suitable $(s,p)$. We give such statements in terms of two well-known measures of smallness: capacity and Hausdorff dimension (Theorem \ref{thm:capacity} and Corollary \ref{cor:hausdorff} respectively). 

Let $K \subset \R^n$ be a compact set. If $s \in (0,1)$ and $p \in [1,\infty)$, recall that the $(s,p)$-capacity of $K$ (c.f. \cite[Definition 2.2.1]{adams_hedberg}) is defined by
\begin{equation} \label{eq:def_cap_wsp}
    \operatorname{Cap}_{s,p}(K) := \inf \big\{ \|u\|_{W^{s,p}(\R^n)}^p: u \in C^\infty_c(\R^n) , u \geq 1 \text{ on } K \big\}.
\end{equation}
Furthermore, if $\Omega \subset\R^n$ is a bounded Lipschitz domain with $K \subset \Omega$, $\mathrm{Cap}_{s,p}(K) = 0$ if and only if there exists a sequence $(\chi_k)_{k=1}^\infty \subset C^\infty_c(\Omega)$ such that
\begin{equation*}
     0 \leq \chi_k \leq 1, \chi_k \geq 1 \text{ in a neighbourhood of } K, \|\chi_k\|_{W^{s,p}(\Omega)} \to 0.
\end{equation*}
This follows from, e.g., \cite[Lemma 5.3]{diNezza_12} and \cite[Theorem 2.1]{shi_xiao_16}.

\begin{proof}[Proof of Theorem \ref{thm:capacity}] Fix $\psi \in C^\infty_c(U')$. We need to show that
    \begin{equation}
        \Det D^2 f[\psi] = \int_{U'} K_g(\Phi)(1+|Df|^2) \psi ~dx.
        \label{Eq:1}
    \end{equation}
    By Theorem \ref{thm:Gauss_23}, we have
    \begin{equation*}
        \Det D^2 f = K_g(\Phi)(1+|Df|^2)^2 \quad \text{ in } \mathcal{D}'(U'\setminus S'),
    \end{equation*}
    where $S' := \Psi(S)$ and $u = (\Psi,v)$ as earlier. Therefore,
    \begin{equation}
        \Det D^2 f[\varphi \psi] = \int_{U'} K_g(\Phi)(1+|Df|^2)\varphi \psi ~dx \quad \forall \varphi \in C^\infty_c(U'\setminus S').
        \label{Eq:2}
    \end{equation}
    To obtain \eqref{Eq:1}, we proceed to choose in \eqref{Eq:2} a suitable sequence $\varphi_k \in C^\infty_c(U'\setminus S')$ which suitably approximates the constant function $1$. This is facilitated by the assumption that $S$ has zero capacity.
    
    Set 
    \begin{equation} \label{eq:s0p0}
        s_0 := 2(1-s); \quad p_0 := \frac{p}{p-2}.
    \end{equation}
    Since $\Psi$ is Lipschitz, we have by {\cite[Theorem 5.2.1]{adams_hedberg}} that
    \begin{equation*}
        \operatorname{Cap}_{s_0,p_0}(S') \leq C(s,p,[\Psi]_{C^{0,1}}) \operatorname{Cap}_{s_0,p_0}(S) = 0.
    \end{equation*}
    We can therefore choose a sequence $(\chi_k) \subset C^\infty_c(U')$ such that $0\leq \chi_k \leq 1$, $\chi_k \equiv 1$ in an open neighbourhood of $S'$, and $\chi_k \to 0$ in $W^{s_0,p_0}(U')$. Let $\varphi_k := 1-\chi_k$. Then,
    \begin{equation*}
        \varphi_k \to 1 \text{ in } W^{s_0,p_0}(U')
    \end{equation*}
    and $\varphi_k \equiv 0$ in an open neighbourhood of $S$ for each $k$. In particular,
    \begin{equation*}
        \varphi_k \psi \to \psi \text{ in } W^{s_0,p_0}_0(U').
    \end{equation*}

    We are now in position to wrap up the argument: By well-known Jacobian estimates (see, e.g. \cite[Proposition 1]{gladbach_olbermann}), we have
    \begin{equation}
        \operatorname{Det} D^2f \in (W^{s_0,p_0}_0(U'))^*.
    \end{equation}
    Hence, using $\varphi = \varphi_k$ in \eqref{Eq:2} and sending $k \rightarrow \infty$, we obtain
    \begin{align*}
        \operatorname{Det} D^2f [\psi] &= \lim_{k \to \infty} \operatorname{Det} D^2 f [\varphi_k \psi]\\
        &= \lim_{k \to \infty} \int_{U'} K_g(\Phi) (1+|Df|^2)^2 \varphi_k \psi~dx\\
        &= \int_{U'} K_g(\Phi) (1+|Df|^2)^2 \psi~ dx.
    \end{align*}
    The proof is complete.
\end{proof}
It is well known that capacity is related to Hausdorff measure. 
\begin{corollary} \label{cor:hausdorff}
    Let $U, g, s, p, f, U'$ be as in Theorem \ref{thm:capacity} and $\sigma := \frac2{p-2} (sp-2)$. If $S \subset U$ a compact set with finite $\sigma$-dimensional Hausdorff measure, then $f$ satisfies the Gauss equation \eqref{eq:gauss2} in all of $U'$.
\end{corollary}

\begin{proof}
    By {\cite[Theorem 5.1.9]{adams_hedberg}}, $\mathcal{H}^\sigma(S) < \infty$ implies that the $(s_0,p_0)$ capacity of $S$ is zero, with $s_0, p_0$ as in \eqref{eq:s0p0}.
\end{proof}

\section{Positive curvature: convexity and regularity} \label{sec:positive}

In this section, we apply the Gauss equation of Theorem \ref{thm:Gauss_23} to prove convexity and regularity results for $C^1 \cap W^{1+\frac23,3}$ isometric immersions of positively-curved smooth surfaces into $\R^3$. The main results are Theorems \ref{thm:rigidity_closed} and \ref{thm:rigidity_caps}, which treat respectively the cases of closed surfaces and convex caps. 

We outline now the key points in the proofs of Theorems \ref{thm:rigidity_closed} and \ref{thm:rigidity_caps}. Recall the important works by Conti, De Lellis, and Székelyhidi \cite{conti_deLellis_szekelyhidi_12} and Pakzad \cite{pakzad_24} where similar results were obtained under different regularity assumptions. In \cite{conti_deLellis_szekelyhidi_12}, the isometric embeddings are assumed to be of class $C^{1,\alpha}(\mathbb{S}^2;\R^3)$ with $\alpha > \frac23$ and the smooth metric $g$ has $K_g > 0$. In \cite{pakzad_24}, the regularity assumption remains the same but the curvature assumption is weakened to $K_g \geq 0$. Using the Gauss equation Theorem \ref{thm:Gauss_23} we are able to adapt their line of argument to our weaker $C^1 \cap W^{1+\frac23,3}$ regularity assumption.
\begin{enumerate}
    \item We apply the Gauss equation from Theorem \ref{thm:Gauss_23} to show that any local graphical representation $f$ satisfies $\Det D^2 f \geq 0$.
    \item We prove a lower bound and integrability result for the local Brouwer degree of $Df$ (see Proposition \ref{prop:deg_when_mu_positive}).
    \item We show that the graph of $f$ has bounded extrinsic curvature and that all regular points are elliptic, in the sense of Pogorelov (see Proposition \ref{prop:graph_bec}). The convexity of the embedded surface then follows from Pogorelov's theory. Theorem \ref{thm:rigidity_closed} follows. We further show that $\Det D^2f$ agrees with the Alexandroff measure. Theorem \ref{thm:rigidity_caps} follows from the regularity theory of Alexandroff solutions for Monge-Ampère equations.
\end{enumerate}

The section is structured as follows. In \S \ref{sec:pogorelov} we recall facts from Pogorelov's theory of surfaces of bounded extrinsic curvature (c.f. \cite{pogorelov_73}) that we will use. In \S \ref{sec:deg_estimates} we prove Proposition \ref{prop:deg_when_mu_positive}. In \S \ref{sec:graphical_surf_is_BEC} we study surfaces given as graphs of scalar functions $f: U \subset \R^2 \to \R$, and we show in Proposition \ref{prop:graph_bec} that such a surface has bounded extrinsic curvature and nonnegative curvature in the Pogorelov sense if $\Det D^2f \geq 0$ as a distribution and $f \in C^1 \cap W^{1+\frac23,3}$. We also show that, subject to a planar boundary condition, the graphical surface is convex. Finally, in \S \ref{sec:proofs_convexity_theorems} we deduce Theorems \ref{thm:rigidity_closed} and \ref{thm:rigidity_caps}.

\subsection{Some facts from Pogorelov's theory} \label{sec:pogorelov}
For the convenience of the reader, we recall a few relevant definitions and statements about surfaces of bounded extrinsic curvature from Pogorelov's monograph \cite{pogorelov_73}, which we will use in the proofs of Theorems \ref{thm:rigidity_closed} and \ref{thm:rigidity_caps}. Readers who are familiar with this theory can safely skip ahead to section \S \ref{sec:deg_estimates}.

To fix terminology, we say that $\Sigma \subset \R^3$ is a $C^1$ embedded surface if for every $x_0 \in \Sigma$, there exists an open neighbourhood $U$ of $x_0$ in $\R^3$ such that $U \cap \Sigma$ is the graph of a $C^1$ function of two variables. Such surfaces need not be complete, and we do not include the boundary in $\Sigma$. There is a continuous map
\begin{equation}
    \nu : \Sigma \to \mathbb{S}^2
\end{equation}
called the Gauss map, unique up to global sign, such that $\nu(x)$ is orthogonal to $T_x\Sigma$ for any $x \in \Sigma$.

In the sequel, $\mathcal{H}^2$ denotes the two-dimensional Hausdorff measure.

\begin{definition}
    A $C^1$ embedded surface $\Sigma \subset \R^3$ has bounded extrinsic curvature if there exists a constant $C < \infty$ such that, for any finite family $\{E_i\}_{i=1}^n$ of pairwise disjoint relatively closed subsets of $\Sigma$,
    \begin{equation*}
        \sum_{i=1}^n \mathcal{H}^2(\nu(E_i)) \leq C.
    \end{equation*}
\end{definition}
\begin{definition}
    Let $\Sigma \subset \R^3$ be a $C^1$ embedded surface with bounded extrinsic curvature. The absolute curvature $\sigma^0$ is defined on subsets of $\Sigma$ as follows. If $O \subset \Sigma$ is open, then
    \begin{equation*}
        \sigma^0(O):= \sup_{E_i \subset O, \text{ pairwise disjoint, closed } }  \sum_{i=1}^N \mathcal{H}^2(\nu(E_i)).
    \end{equation*}
    For $A \subset \Sigma$ arbitrary,
    \begin{equation*}
        \sigma^0(A) := \inf_{ \text{open sets } O \supseteq A} \sigma^0(O).
    \end{equation*}
\end{definition}
The absolute curvature is a Borel measure on $\Sigma$ (it is countably additive \cite[Theorem 2, p. 590]{pogorelov_73}) and it satisfies $\mathcal{H}^2(\nu(A)) \leq \sigma^0(A)$.
\begin{definition} \label{def:reg_pts}
    Let $\Sigma \subset \R^3$ be a $C^1$ embedded surface with bounded extrinsic curvature. A point $x \in \Sigma$ is called regular if it has an open neighbourhood $U_x$ in $\Sigma$ such that $\nu(y) \neq \nu(x)$ for all $y \in U_x \setminus \{x\}$. Regular points $x$ are classified according to the intersection $U_x \cap T_x \Sigma$ of the tangent plane with any sufficiently small open neighbourhood of $x$ in $\Sigma$. We say that $x$ is:
    \begin{enumerate}
        \item \textit{elliptic} if $U_x \cap T_x\Sigma = \{x\}$;
        \item \textit{parabolic} if $U_x \cap T_x \Sigma$ is a union of two simple curves issuing from $x$;
        \item \textit{hyperbolic} if $U_x \cap T_x\Sigma$ is a union of four simple curves issuing from $x$;
        \item \textit{flat}, otherwise, in which case $U_x \cap T_x \Sigma$ must be a union of an even number $2n$, $n \geq 3$, of simple curves issuing from $x$.
    \end{enumerate}
\end{definition}

The classification of regular points is characterised by the index of the Gauss map. If $x \in \Sigma$ is regular, consider any sufficiently small neighbourhood $U$ of $x$ in $\Sigma$, homeomorphic to a disc with boundary curve $\partial U = \gamma$. Then, $\nu(x)$ does not belong to $\nu(\gamma)$ and we define $i(x)$ to be the winding number of $\nu(\gamma)$ about $\nu(x)$. This is an integer which is independent of the choice of such $\gamma$.

\begin{proposition} {\cite[Lemma, p. 594]{pogorelov_73}} \label{prop:index_at_regular_points}
    Let $\Sigma \subset \R^3$ be a surface of bounded extrinsic curvature and $x \in \Sigma$ be a regular point. Then, $i(x)$ is equal to $+1, -1, 0$, or a number $<-1$ if $x$ is, respectively, elliptic, hyperbolic, parabolic, or flat.
\end{proposition}

It is a fact \cite[Theorem 4, p. 591]{pogorelov_73} that the set $\mathcal{N} \subset \Sigma$ of nonregular points has zero absolute curvature: $\sigma^0(\mathcal{N}) = 0$. 

\begin{definition}
    Let $\Sigma \subset \R^3$ be a $C^1$ embedded surface with bounded extrinsic curvature. The positive and negative curvatures $\sigma^\pm$ are defined by
    \begin{equation*}
        \sigma^\pm(A) := \sigma^0(A\cap\Sigma_\pm), \quad A \subseteq \Sigma,
    \end{equation*}
    where $\Sigma_+$ is the set of elliptic points of $\Sigma$ and $\Sigma_-$ is the set of hyperbolic points. We say $\Sigma$ has nonnegative (resp. nonpositive) extrinsic curvature if $\sigma^-(\Sigma) = 0$ (resp. $\sigma^+(\Sigma) = 0$).
\end{definition}

We also record the following fact, which follows from \cite[Theorem 13, p. 601]{pogorelov_73}.
\begin{proposition} \label{prop:elliptic_exists}
    Let $\Sigma$ be a surface of bounded extrinsic curvature homeomorphic to $\mathbb{S}^2$. Then $\Sigma$ contains an elliptic point.
\end{proposition}

The curvatures $\sigma^\pm$ and $\sigma^0$ are also related to the preimages under the Gauss map \cite[Theorem 7, p. 594]{pogorelov_73}:
\begin{equation} \label{eq:curvature_and_indicatrix}
    \sigma^0(A)= \int_{\mathbb{S}^2} n_A(y)~d\sigma(y), \quad \sigma^\pm(A) = \int_{\mathbb{S}^2} n_A^\pm(y)~d\sigma(y)
\end{equation}
where for $y \in \mathbb{S}^2$, $n_A(y) := \#\big(\nu^{-1}(y) \cap A\big)$ and $n_A^\pm(y) := \#\big( \nu^{-1}(y) \cap A \cap \Sigma_\pm \big)$.

By \cite[Theorem 12, p. 600]{pogorelov_73}, if $\Sigma$ contains an elliptic (resp. hyperbolic) point, then $\sigma^+(\Sigma) > 0$ (resp. $\sigma^-(\Sigma) < 0$). The following important facts hold about surfaces of nonnegative curvature. Recall that, in Pogorelov's terminology, a surface $\Sigma \subset \R^3$ is a convex surface if there exists a convex body $K \subset \R^3$ such that $\Sigma \subset \partial K$.

\begin{proposition}{\cite[Theorems 1 and 2, p. 613, 615]{pogorelov_73}} \label{prop:pog_convexity}
Let $\Sigma$ be a surface of bounded extrinsic curvature, of nonnegative curvature.
\begin{enumerate}
    \item[(i)] If $\Sigma$ is complete and $\sigma^+(\Sigma) > 0$, then $\Sigma$ is either a closed convex surface or an unbounded convex surface. 
    \item[(ii)] Suppose that $\Pi$ is a plane intersecting the surface. If $\Sigma'$ is any component of $\Sigma \setminus\Pi$ such that $\partial \Sigma' \subset \Pi$, then $\Sigma'$ is a convex surface.
\end{enumerate}    
\end{proposition}

\begin{remark} \label{rem:pog_convex_cap}
    A careful study of Proposition \ref{prop:pog_convexity}(ii) and Pogorelov's proof shows that the statement may be extended to the case of surfaces with boundary as follows. Let $\Sigma \subset \R^3$ be an open bounded surface of bounded extrinsic curvature, of nonnegative curvature. Assume that the boundary $\Gamma := \partial \Sigma$ of $\Sigma$ (as a subset of $\R^3$) lies in a plane $\Pi$, which we take for convenience to be $\{x_3 = 0\}$. Assume further that  $\Sigma$ lies strictly on one side of $\Pi$, e.g. $\Sigma \subset \{x_3 < 0\}$. Then, $\Sigma$ is a convex surface. This follows from Pogorelov's proof, with the only observation that we now need to use perturbations of $\Pi$ by planes whose intersection with $\Sigma$ is closed in $\R^3$ and lies entirely below $\Pi$. 
\end{remark}

\subsection{Estimates for the degree of $Df$} \label{sec:deg_estimates}

Recall that for an open set $V \subset \R^n$ and a continuous map $v \in C^0(\overline V;\R^n)$ the Brouwer degree $\deg(v,V,y)$ is well-defined for $y \in \R^n \setminus v(\partial V)$. If $v \in C^1(\overline V)$ and $y \in \R^n \setminus v(\partial V)$ is a regular value of $v$, then it agrees with the classical definition
\begin{equation*}
    \deg(v,V,y) := \sum_{x \in v^{-1}(\{y\})} \operatorname{sign}(\det Dv(x));
\end{equation*}
in the general case that $v$ is continuous, $\deg(v,V,y)$ is defined via approximation. It takes integer values and is constant on each connected component of $\R^n \setminus v(\partial V)$. In particular, it is integrable over any compact set $K \subset \R^n \setminus v(\partial V)$ since $K$ intersects only finitely many such connected components. 

The main point of this section is the following Proposition \ref{prop:deg_when_mu_positive}, which establishes a lower bound \eqref{eq:deg_geq_1} and integral estimate \eqref{eq:deg_in_L1} for the degree of $Df$ when $f : U \subset \R^2 \to \R$ is regular enough and $\Det D^2 f \geq 0$.

\begin{proposition} \label{prop:deg_when_mu_positive}
    Let $\Omega \subset \R^2$ be a domain and assume that $f \in (W^{1+\frac23,3}_{loc} \cap C^1)(\Omega)$ satisfies $\Det D^2 f = \mu$ in $\mathcal{D}'(\Omega)$, where $\mu \in \mathcal{M}^+(\Omega)$. Then, for any open set $V \subset \subset \Omega$, the local degree of $Df$ is positive and integrable over $\R^2 \setminus Df(\partial V)$:
    \begin{equation} \label{eq:deg_geq_1}
        \deg(Df, V,\cdot) \geq 1_{Df(V) \setminus Df (\partial V)},
    \end{equation}
    \begin{equation} \label{eq:deg_in_L1}
        \int_{\R^2 \setminus Df(\partial V)} \deg(Df, V,y)~dy\leq \mu(V).
    \end{equation}
\end{proposition}

\begin{remark} \label{rem:deg_formula_nobdry}
    Note that $Df(\partial V)$ is not necessarily an $\mathcal{L}^2$-null set and must thus be excluded from the region of integration in \eqref{eq:deg_in_L1} as the degree $\deg(Df, V, \cdot)$ is not defined there. However, if we have in addition that $\partial V$ is piecewise smooth and $Df \in W^{\frac23,3}(\partial V)$, then $Df(\partial V)$ is $\mathcal{L}^2$-null \cite[Theorem B.1]{li_pakzad_schikorra_21} and \eqref{eq:deg_in_L1} can be restated as 
    \begin{equation*}
        \int_{\R^2} \deg(Df, V,y)~dy\leq \mu(V).
    \end{equation*}
\end{remark}

To prove Proposition \ref{prop:deg_when_mu_positive}, we use a trick due to Kirchheim \cite{kirchheim}: Instead of working directly with $Df$, we consider a slight perturbation $v^\delta = Df + \delta(-x_2,x_1)$ and send $\delta \rightarrow 0$ afterwards. We need the following variant of \cite[Lemma 3.1]{li_pakzad_schikorra_21} (c.f. \cite[Corollary 5]{pakzad_24}).

\begin{lemma} \label{lemma:Det_and_deg}
    Let $\Omega \subset \R^2$ be a domain, and assume that $f \in W^{1+\frac23,3}_{loc}\cap C^1(\Omega)$. For $\delta \in \R$ define $v^\delta(x):=Df(x) + \delta(-x_2,x_1)$. Let $V \subset \subset \Omega$ be any open set compactly contained in $\Omega$. Then, there holds
    \begin{equation} \label{eq:degree_formula}
        (\Det D^2 f)|_V  [\phi \circ v^\delta] + \delta^2 \int_V \phi (v^\delta(x))dx= \int_{\R^2} \phi(y) \deg(v^\delta, V,y)dy.
    \end{equation}
    for any $\phi \in C^\infty_c(\R^2 \setminus v^\delta(\partial V))$.
\end{lemma}

\begin{proof}
    Fix any $V \subset\subset \Omega$ and $\phi \in C^\infty_c(\R^2 \setminus v^\delta(\partial V))$. By \cite[Lemma 3.1]{li_pakzad_schikorra_21} and the regularity assumption on $f$, there holds
    \begin{equation*}
        \Det Dv^\delta|_V [\phi \circ v^\delta] = \int_{\R^2} \phi(y) \deg(v^\delta, V, y)dy.
    \end{equation*}
    Moreover, by approximating $f$ with smooth functions it is easy to see that
    \begin{equation*}
        \Det Dv^\delta = \Det D^2f + \delta^2 \quad \text{in } \mathcal{D}'(U).
    \end{equation*}
    The result follows.
\end{proof}

\begin{proof}[Proof of Proposition \ref{prop:deg_when_mu_positive}]
    For $\delta > 0$, let $v^\delta(x):=Df(x) + \delta(-x_2,x_1)$. Recall that $\deg(v^\delta, V,\cdot)$ is integer-valued and constant on each connected component of $\R^2 \setminus v^\delta(\partial V)$. Since $\Det D^2 f \geq 0$, it follows from Lemma \ref{lemma:Det_and_deg} that $\deg(v^\delta, V,\cdot) \geq 0$ for any $\delta > 0$. We claim further that if $y \in v^\delta(V) \setminus v^\delta(\partial V)$ then $\deg(v^\delta, V, y) \geq 1$. Indeed, for any such $y$, we may take a ball $B_r(y) \subset \R^2 \setminus v^\delta(\partial V)$ with $r > 0$ and a cutoff $\phi \in C^\infty_c(B_r(y))$ with $\phi \geq 0$ and $\phi \equiv 1$ in $B_{r/2}(y)$. Then, Lemma \ref{lemma:Det_and_deg} and the constancy of $\deg(v^\delta, V, \cdot)$ in $B_r(y)$ imply that
    \begin{align*}
        \deg(v^\delta, V,y) \int_{\R^2} \phi(z) dz &\geq \delta^2 \int_V \phi(v^\delta(x))dx\\
        &\geq \delta^2 |\{x \in V : v^\delta(x) \in B_{r/2}(y)\}|.
    \end{align*}
    Since $\{x \in V : v^\delta(x) \in B_{r/2}(y)\}$ is open (since $v^\delta$ is continuous) and non-empty (since it contains the preimage of $y$ in $V$), it has positive measure. Therefore, $\deg(v^\delta, V,y) > 0$. Since $\deg(v^\delta,V,y)$ is an integer, it follows that $\deg(v^\delta, V,y) \geq 1$, as claimed.
    
    We may now follow the proof of \cite[Corollary 5]{pakzad_24} to send $\delta \searrow 0$. For completeness, we give the details. Let $y = Df(x) \in Df(V) \setminus Df(\partial V)$. Fix $r > 0$ such that $B_r(y) \subset \R^2 \setminus Df(\partial V)$. Since $v^\delta$ converges uniformly to $Df$ on $\overline V$, we may fix $\delta_0>0$ such that $\|v^\delta - Df\|_{C^0(\overline V; \R^2)} < r/2$ for all $\delta \in (0,\delta_0)$. Then,
    \begin{equation*}
        B_{r/2}(y) \subset \R^2 \setminus v^\delta(\partial V) \quad \text{for any } 0 < \delta < \delta_0,
    \end{equation*}
    and so $\deg(v^\delta, V,\cdot)$ is well-defined and constant on $B_{r/2}(y)$ for all $\delta \in (0, \delta_0)$. In particular, for all such $\delta$, we have $v^\delta(x) \in B_{r/2}(y)$. Therefore,
    \begin{equation*}
        \deg(v^\delta, V,y) = \deg(v^\delta, V,v^\delta(x)) \geq 1.
    \end{equation*}
    Moreover, by the uniform convergence $v^\delta \to Df$ in $\overline V$ as $\delta \searrow 0$, we have $\deg(v^\delta, V,y) = \deg(Df, V,y)$ for all sufficiently small $\delta$. Hence, $\deg(Df,V,y) \geq 1$ and \eqref{eq:deg_geq_1} is proved.

     To see \eqref{eq:deg_in_L1}, we take a sequence $(\phi_k) \subset C^\infty_c(\R^2 \setminus Df(\partial V))$ such that $\phi_k \geq 0$ and $\phi_k \nearrow 1_{\R^2 \setminus Df(\partial V)}$ pointwise. For each $k$, Lemma \ref{lemma:Det_and_deg} with $\delta = 0$ and $\phi = \phi_k$ gives that
    \begin{equation*}
        \int_{\R^2 \setminus Df(\partial V)} \phi_k(y) \deg(Df, V,y) dy = \int_V \phi_k(Df(x)) d\mu(x) \leq \mu(V).
    \end{equation*}
    By the monotone convergence theorem, we conclude with \eqref{eq:deg_in_L1}.
\end{proof}

\subsection{Bounded extrinsic curvature and ellipticity of regular points} \label{sec:graphical_surf_is_BEC}

\begin{proposition} \label{prop:graph_bec}
    Let $\Omega \subset \R^2$ be an open set, and $\mu \in \mathcal{M}^+(\Omega)$. Assume that $f \in C^1 \cap W^{1+\frac23,3}_{loc}(\Omega)$ satisfies 
    \begin{equation*}
        \Det D^2 f = \mu.
    \end{equation*}
    Then, if $\mu(\Omega) < \infty$, there holds
    \begin{equation} \label{eq:alex_leq_mu_closed}
        |Df(A)| \leq \mu(A)
    \end{equation}
    for any relatively closed or open subset $A \subset \Omega$. The surface $\Sigma:= G_f(\Omega)$ has bounded extrinsic curvature. Moreover, all its regular points are elliptic, so $\Sigma$ has nonnegative curvature in the sense of Pogorelov.
\end{proposition}

If we knew a priori that $f$ is convex, then $|Df(\cdot)|$ is the Alexandroff measure and agrees with $\mu$ and there is nothing to prove. We emphasise that in the proposition we only assume the non-negativity of the distributional Hessian determinant, rather than the convexity of $f$.

\begin{proof}
    The proof proceeds similarly to those of \cite[Proposition 3.1]{pakzad_24} and \cite[Theorem 3]{conti_deLellis_szekelyhidi_12} with some necessary adaptations. 
 
 \medskip
 \noindent\underline{Step 1.} We claim that if $V \subset\subset \Omega$ is an open set with $|Df(\partial V)| = 0$, then $
            |Df(V)| \leq \mu(V)$.
            
        Note that as $V$ is an increasing union of compact sets, $Df(V)$ is Borel-measurable. Since $Df(\partial V)$ is $\mathcal{L}^2$-null, Proposition \ref{prop:deg_when_mu_positive} and Remark \ref{rem:deg_formula_nobdry} give
        \begin{align*}
            |Df(V)| &= \int_{\R^2} 1_{Df(V) \setminus Df(\partial V)}~dy\\
            &\leq \int_{\R^2} \deg(Df, V,y)~dy\\
            &\leq \mu(V).
        \end{align*}

\medskip
 \noindent\underline{Step 2.} Next we claim that $|Df(K)| \leq \mu(K)$ for all compact sets $K \subset \Omega$. 
 
 Fix $\delta \in (0, \dist(K,\partial \Omega))$. By a Fubini-type theorem in fractional Sobolev spaces (see, e.g. \cite[Lemma 2.2]{li_schikorra}) we know that for every $x \in K$ there exists $r_x \in (0,\delta)$ such that $Df \in W^{\frac23,3}(\partial B_{r_x} (x))$ and in particular, $Df(\partial B_{r_x}(x))$ is $\mathcal{L}^2$-null \cite[Theorem B.1]{li_pakzad_schikorra_21}. Cover $K$ by finitely many such open balls: $K \subset \cup_{i=1}^N B_i$. Define
        \begin{equation*}
            \Gamma := \bigcup_{i=1}^N \partial B_i, \quad W := \bigcup_{i=1}^N B_i.
        \end{equation*}
        Then, $W \setminus \Gamma$ is a finite disjoint union of connected open sets $W_i$. Each boundary $\partial W_i$ is a union of finitely many circular arcs contained in $\Gamma$. Therefore, by step 1 applied to the $W_i$,
        \begin{align*}
            |Df(K)| \leq \sum_i |Df(W_i)| \leq \sum_i \mu(W_i) \leq \mu(K + B_\delta).
        \end{align*}
        In the above, it was crucial that the $W_i$ are disjoint and that each $Df(\partial W_i)$ is $\mathcal{L}^2$-null. Since $K + B_\delta$ decreases to $K$ and $\mu(K) < \infty$ we now let $\delta \searrow 0$ to conclude that $|Df(K)| \leq \mu(K)$.

\medskip
 \noindent\underline{Step 3.} If $A \subset \Omega$ is a relatively closed or open set, then $|Df(A)| \leq \mu(A)$. 
 
 To see this, observe that in either case, $A$ is a countable increasing union of compact sets $K_i$. Therefore, $Df(A)$ is Lebesgue-measurable and we pass to the limit $i \to \infty$ in the inequality $|Df(K_i)| \leq \mu(K_i) \leq \mu(E)$ from step 2 to prove the claim.

\medskip
 \noindent\underline{Step 4.}  Now, we prove that $\Sigma$ has bounded extrinsic curvature.
 
Let $\{E_i\}_{i=1}^N$ be any finite family of pairwise disjoint relatively closed subsets of $\Sigma$. Let $\pi$ denote the projection $\pi(x_1,x_2,x_3):=(x_1,x_2)$. Then, $\pi$ is a bijection from $\Sigma$ to $\Omega$ with inverse $G_f$. In particular $\pi : \Sigma \to \Omega$ is a homeomorphism. It follows that $A_i := \pi(E_i)$ are pairwise disjoint relatively closed subsets of $\Omega$ with $E_i = G_f(A_i)$. By step 3, we have
        \begin{equation*}
            \sum_{i=1}^N |Df(A_i)| \leq \sum_{i=1}^N \mu(A_i) \leq \mu(\Omega) < \infty.
        \end{equation*}
        Next, observe that the unit normal $\nu : \Sigma \to \mathbb{S}^2$ satisfies
        \begin{equation*}
            \nu \circ G_f = \xi\circ Df, \quad \xi(z):=(1+|z|^2)^{-\frac12}(-z,1)
        \end{equation*}
        and that $\xi$ is a diffeomorphism from $\R^2$ onto the upper hemisphere of $\mathbb{S}^2$. By a direct calculation, $|\partial_1 \xi \times \partial_2 \xi| \leq 1$ in $\R^2$; thus, the area formula yields
        \begin{equation*}
            \mathcal{H}^2(\nu(E_i)) = \int_{Df(A_i)} |\partial_1 \xi \times \partial_2 \xi|dx \leq |Df(A_i)|.
        \end{equation*}
        for each $i=1,\dots,N$. It follows that
        \begin{equation*}
            \sum_{i=1}^N \mathcal{H}^2(\nu(E_i)) \le \mu(\Omega).
        \end{equation*}
        Thus, $\Sigma$ has bounded extrinsic curvature. 

\medskip
 \noindent\underline{Step 5.} We check that all regular points are elliptic. 
 
 By the definition of the index and the fact that $\nu \circ G_f = \xi \circ Df$, we have
            \begin{equation*}
                i(p) = \deg(Df, B_r(x), Df(x))
            \end{equation*}
            for any $r>0$ sufficiently small (for a detailed justification, see, e.g. the proof of \cite[Proposition 3.1]{pakzad_24}). It follows then from \eqref{eq:deg_geq_1} that $i(p) \geq 1$, and hence by Proposition \ref{prop:index_at_regular_points} that $p$ is elliptic.
\end{proof}

We next show that if, in addition to the hypotheses of Proposition \ref{prop:graph_bec}, $f$ satisfies the boundary condition $f = 0$ on $\partial \Omega$ and the sign assumption $f < 0$ in $\Omega$, then $f$ is convex and $\mu$ agrees with the Alexandroff measure of $f$.

Recall that if $\Omega \subset \R^n$ is a convex domain and $f : \Omega \to \R$ is a convex function, the Alexandroff measure of $f$ is
\begin{equation*}
    \mu_f(A) := \mathcal{L}^n(\partial f(A)),
\end{equation*}
where $\partial f$ denotes the subgradient of $f$; this gives a well-defined positive Borel measure \cite{figalli_19}. 

\begin{proposition} \label{prop:graph_alexandroff}
    Let $\Omega \subset \R^2$ be a bounded domain, and $\mu$ a nonnegative finite Borel measure on $\Omega$. Assume that $f \in W^{1+\frac23,3}_{loc}(\Omega) \cap C^1(\Omega) \cap C^0(\overline \Omega)$ satisfies
\begin{equation*} 
\Det D^2 f = \mu \quad \text{in} \quad \mathcal{D}'(\Omega) 
\end{equation*}
together with $f = 0 \text{ on } \partial \Omega$ and $f < 0$ in a neighbourhood of some connected component $\gamma_1$ of $\partial \Omega$. Then, $\Sigma:=G_f(\Omega)$ is a convex surface. Hence, $\Omega$ is a convex domain and $f$ is convex on $\Omega$ with Alexandroff measure $\mu$.
\end{proposition}

\begin{remark}
    The convexity of $f \in C^{1,\alpha}(\Omega)$ with $\Det D^2 f \geq 0$ has been studied in \cite{pakzad_24, lewicka_pakzad_17}. It is not yet known, even for $\alpha > \frac23$ and $\Det D^2 f \geq \delta > 0$, whether convexity holds without imposing the planar boundary condition and one-sided condition above. On the other hand, for $\alpha < \frac13$ there exist abundant nonconvex such $f$ \cite{cao_hirsch_inauen_MA_25, cao_szekelyhidi_19, lewicka_pakzad_17}.
\end{remark}

\begin{proof}
    We begin by showing that $\Omega$ is a convex domain and that $f$ is convex on $\Omega$. Let $\Omega_- := \{x \in \Omega : f(x) < 0\}$. For each component $\Omega_i$ of $\Omega_-$, the graph $\Sigma_i := G_f(\Omega_i)$ of $f$ below $\Omega_i$ is a bounded, $C^1$ embedded open surface in $\R^3$, whose boundary lies in the plane $\{x_3 = 0\}$. By Proposition \ref{prop:graph_bec} applied to $f$ on $\Omega_-$, the graphical surface $\Sigma_- := G_f(\Omega_-)$ is a (possibly disconnected) surface of bounded extrinsic curvature and of nonnegative curvature. Hence, by Remark \ref{rem:pog_convex_cap} following Pogorelov's Proposition \ref{prop:pog_convexity}, each component $\Sigma_i$ of $\Sigma_-$ is a convex surface. Therefore, to prove the convexity of $\Omega$ and $f$ it suffices to show that $\Omega$ coincides with a connected component of $\Omega_-$.
    
    Without loss of generality, let $\Omega_1$ be a connected component of $\Omega_-$ whose boundary intersects $\gamma_1$. By the assumption that $f < 0$ in a neighbourhood of $\gamma_1$, such $\Omega_1$ exists. (It can be shown that $\Omega_1$ is unique and that $\partial \Omega_1$ contains $\gamma_1$ but we will not need these facts.) Let $\tilde\gamma_1 := \partial \Omega_1$ denote the boundary of $\Omega_1$. Since $\Sigma_1 = G_f(\Omega_1)$ is a bounded convex surface with boundary $\tilde\gamma_1$ on the plane $\{x_3 = 0\}$, we know that $\tilde\gamma_1$ is the intersection of a convex body $K$ with the plane $\{x_3 = 0\}$. Hence, $\tilde\gamma_1$ is either empty, a point, or a closed simple convex plane curve. Since $\tilde\gamma_1$ is the boundary of a nonempty bounded domain $\Omega_1$ in $\R^2$, it contains infinitely many points\footnote{Fix $x_0 \in \Omega_1$. For each direction $\xi \in \mathbb{S}^1$, let $r_\xi := \sup \{r > 0: x_0 + r\xi \in \Omega_1\}$. Then $x_0 + r_\xi \xi \in \partial \Omega_1$, so there is an injection from $\mathbb{S}^1$ into $\partial \Omega_1$.}. Hence, $\tilde\gamma_1$ is a closed simple convex plane curve. 

    We next claim that $\tilde\gamma_1$ is contained in $\gamma_1$. If not, then since $\tilde\gamma_1$ is a continuous curve which intersects $\gamma_1$ at some point $p$, $\tilde\gamma_1$ would have to pass through the neighbourhood of $\gamma_1$ in $\Omega$ where $f < 0$. But since $\tilde\gamma_1 \subset \partial \Omega_-$ there holds $f = 0$ along $\tilde\gamma_1$, a contradiction. Hence, $\tilde\gamma_1 \subset \gamma_1$.

    Now we argue that $\Omega = \Omega_1$, which implies that $\Omega$ is a convex domain and that $f$ is convex on $\Omega$. Note that $\Omega$ is a connected open subset of $\mathbb{R}^2 \setminus \gamma_1$, so that $\Omega \subset \mathbb{R}^2 \setminus \tilde\gamma_1$ by the previous claim. Since $\tilde\gamma_1$ is a closed simple curve, by the Jordan curve theorem $\Omega$ is contained either in the inside of $\tilde\gamma_1$ (namely, $\Omega_1$) or the (unbounded) outside of $\tilde\gamma_1$. Because $\Omega_1 \subseteq \Omega$, the former case holds, and $\Omega = \Omega_1$.
    
    Finally, since $f$ is convex and continuously differentiable, it is well known that its Alexandroff measure $\mu_f$ coincides with its distributional Hessian determinant $\Det D^2 f$. For completeness, we recall the argument here. Fix a test function $\psi \in C^\infty_c(\Omega)$ with support in $V \subset \subset \Omega$ and consider the mollifications $f_\epsilon := f*\rho_\epsilon$ with $0 < \epsilon < \epsilon_0:= \dist(V,\partial \Omega)$. Since $f$ is convex in $\Omega$, $f_\epsilon$ is convex in $V$ for any $\epsilon \in (0,\epsilon_0)$. Since $f_\epsilon \to f$ locally uniformly in $\Omega$, the Alexandroff measures converge weakly-star: $\mu_{f_\epsilon} \overset \ast \rightharpoonup \mu_f$ \cite[Proposition 3.2]{figalli_19}. Moreover, $f_\epsilon \to f$ in $H^1_{loc}(\Omega)$, so the distributional Hessian determinants converge. Thus,
    \begin{equation*}
        \Det D^2 f[\psi] = \lim_{\epsilon \searrow 0} \int_\Omega \det D^2 f_\epsilon \psi~dx = \lim_{\epsilon \searrow 0} \int_\Omega \psi~d\mu_{f_\epsilon} = \int_\Omega \psi ~d\mu_f.
    \end{equation*}
\end{proof}

\subsection{Proofs of the rigidity theorems} \label{sec:proofs_convexity_theorems}

\begin{proof}[Proof of Theorem \ref{thm:rigidity_closed}]
    Any point $u(x) \in \Sigma$ has a neighbourhood $V$ in $\Sigma$ which is the graph of a scalar function $f \in C^1(\overline B) \cap W^{1+\frac23,3}(B)$ for some ball $B \subset \R^2$. Moreover, Theorem \ref{thm:Gauss_23} implies that $f$ satisfies the Gauss equation \eqref{eq:gauss2} on $B$. Since $K_g \geq 0$, Proposition \ref{prop:graph_bec} says that $V := G_f(B)$ has bounded extrinsic curvature and is of nonnegative curvature. Covering the compact surface $\Sigma$ with finitely many such patches $V$, we conclude that $\Sigma$ has bounded extrinsic curvature and nonnegative curvature. We claim that the positive curvature of $\Sigma$ is nonzero. By \cite[Theorem 12, p. 600]{pogorelov_73} it is enough to show that there is an elliptic point, which we know from Proposition \ref{prop:elliptic_exists}. Hence, by Proposition \ref{prop:pog_convexity}(i), $\Sigma$ is then a closed convex surface. Such surfaces are known to be rigid in $\R^3$ \cite[Theorem 1, p. 167]{pogorelov_73}. Moreover, by the existence part of the Weyl problem (c.f. \cite[Chapter 9]{han_hong_06}), if $K_g > 0$, then there exists a $C^\infty$ isometric embedding $u_0 : (\mathbb{S}^2,g) \to \R^3$. Since $u_0$ differs from $u$ by a rigid motion, $u$ is also $C^\infty$.
\end{proof}

\begin{proof}[Proof of Theorem \ref{thm:rigidity_caps}]
    Composing $u$ with a rigid motion if necessary, we may assume that $\Pi = \{x_3=0\}$ and that $u^3 < 0$ in $\mathcal{N}$. The assumptions on $u$ imply that $u(U)$ is globally the graph of $f:= v\circ \Phi$, where $u = (\Psi, v)$ and $\Phi := \Psi^{-1}$ as before. Denoting again $U' := \Psi(U)$ we have $f \in W^{1+\frac23,3}_{loc}(U') \cap C^1(\overline U')$, and by Theorem \ref{thm:Gauss_23} $f$ satisfies
    \begin{equation} \label{eq:gauss_for_embedding}
        \Det D^2 f = K_g(\Phi)(1+|Df|^2)^2 \quad \text{in } \mathcal{D}'(U').
    \end{equation}
    with the conditions
    \begin{equation*}
        f = 0 \text{ on } \partial U', \quad f < 0 \text{ in } \Psi(\mathcal{N}) \subset U'.
    \end{equation*}
    
    In the case that $K_g \geq 0$, Proposition \ref{prop:graph_alexandroff} implies that $f$ is convex, $U'$ is a convex domain, $f$ solves \eqref{eq:gauss_for_embedding} in the Alexandroff sense, and $u(\overline U)$ is a convex surface.
    
    When $K_g > 0$, the smoothness of $f$ follows from well known theory of the Monge-Amp\`ere equation: Indeed, since $f \in C^1(\overline U')$ and there exist $\lambda, \Lambda > 0$ such that $\lambda \leq K_g \leq \Lambda$ in $U$, the right-hand side of \eqref{eq:gauss_for_embedding} has uniform positive upper and lower bounds in $U'$. Because the dimension is $n=2$ and $\det D^2f \geq \lambda$ in the Alexandroff sense, $f$ is strictly convex (see \cite{alexandroff}). By \cite[Theorem 2]{caffarelli_91}, $f \in C^{1,\alpha}_{loc}(U')$ for some $\alpha \in (0,1)$. Then, since the right-hand side of \eqref{eq:gauss_for_embedding} is in $C^{0,\alpha}_{loc}(U')$, \cite[Theorem 2]{caffarelli_90} implies that $f \in C^{2,\alpha}_{loc}(U')$. Boostrapping further yields smoothness of $f$, hence $u$.
\end{proof}

\appendix
\section{Comparison of Theorem \ref{thm:Gauss_23} with a result of Conti, De Lellis, and Székelyhidi} \label{sec:delellis_vs_MA}

We discuss now the relationship between Theorem \ref{thm:Gauss_23} and the version \eqref{eq:delellis_gauss} of the Gauss equation from Conti -- De Lellis -- Székelyhidi \cite[Proposition 6]{conti_deLellis_szekelyhidi_12}. 
\begin{proposition} \label{prop:CdLS_vs_Gauss}
    Let $U$, $g$, $u$, $f$ be as in case \eqref{case:Gauss23} of Theorem \ref{thm:Gauss_23}. Then \eqref{eq:delellis_gauss} is true.
\end{proposition}

\begin{proof}
    Let $V \subset\subset U$ and $\phi \in C^\infty_c(\mathbb{S}^2 \setminus \nu(\partial V))$. As in the proof of Proposition \ref{prop:graph_bec}, let $\xi : \R^2 \to \mathbb{S}^2$ be given by $\xi(z) := (1+|z|^2)^{-\frac12}(-z,1)$ so that $\nu = \xi \circ Df$ on $U'$. Recall that $\xi$ is a diffeomorphism onto the upper hemisphere, and it pulls back the volume form $d\sigma$ on $\mathbb{S}^2$ via 
    \begin{equation*}
        \xi^*(d\sigma) = |\partial_1 \xi \times \partial_2 \xi|~dz^1 \wedge dz^2 = (1+|z|^2)^{-\frac32} dz^1 \wedge dz^2.
    \end{equation*}
    Letting $V' := \Psi(V) \subset \subset U'$, changing variables via $\Phi$, and using \eqref{eq:gauss2} we have
    \begin{align*}
        \int_V K_g(x)\phi(\nu(x))dvol_g(x) &= \int_{V'} K_g(\Phi(x)) \phi\big(\xi(Df(x))\big) (1+|Df|^2)^{\frac12} dx\\
        &= \Det D^2f|_{V'} [\psi \circ Df]
    \end{align*}
    where we defined
    \begin{equation*}
        \psi := (\phi \circ \xi ) |\partial_1\xi \times \partial_2 \xi| \in C^\infty_c(\R^2 \setminus Df(\partial V')).
    \end{equation*}
    By Lemma \ref{lemma:Det_and_deg}, we have
    \begin{align*}
        \Det D^2f|_{V'}[\psi\circ Df] = \int_{\R^2 \setminus Df(\partial V')} \psi(z) \deg(Df, V', z) dz.
    \end{align*}
    Furthermore, one has
    \begin{equation*}
        \deg(Df, V',y) = \deg(\nu,V,\xi(y))
    \end{equation*}
    so we pull back via $\xi$ to find
    \begin{align*}
        \int_{\R^2 \setminus Df(\partial V')} \psi(z) \deg(Df, V',z)dz &= \int_{\mathbb{S}^2 \setminus \nu(\partial V)} \psi(\xi^{-1}(y))\deg(\nu,V,y) ~(\xi^{-1})^*dz\\
        &= \int_{\mathbb{S}^2 \setminus \nu(\partial V)} \phi(y) \deg(\nu,V,y) ~d\sigma(y).
    \end{align*}
    Therefore, \eqref{eq:delellis_gauss} holds. 
\end{proof}

\section{Cone-type example violating the Gauss equation} \label{sec:examples}

In this appendix we show that there exists an isometric embedding $u : B \subset \R^2 \to \R^3$ of the Euclidean metric with $ u \in W^{2,p} \cap C^{0,1}(B;\R^3) \cap C^\infty(B \setminus \{0\};\R^3)$ for any $1 \leq p < 2$, such that the embedded surface is the graph of a function $f$ defined on a planar domain with $\Det D^2f$ equal to a nonzero multiple of a Dirac mass. Our construction is inspired by \cite{liu_maly}. More specifically, we consider isometric embeddings $u : B \subset \R^2 \to \R^3$ given in polar coordinates by
\begin{equation} \label{eq:u_from_gamma}
        u(r\cos\theta,r\sin\theta) := r\gamma(\theta),
\end{equation}
where $\gamma: \mathbb{S}^1 \to \mathbb{S}^2$ is a smooth unit-speed simple closed curve. Note that $\gamma$ has length $2\pi$. It is a routine calculation to show that $u \in W^{2,p}\cap C^{0,1}(B;\R^3) \cap C^\infty(B \setminus \{0\};\R^3)$ for any $1 \leq p < 2$ and $u$ is an isometric immersion of the Euclidean metric, i.e. satisfies $Du^T Du = I$ almost everywhere. Geometrically, the embedded surface is the union of all straight line segments in $\R^3$ joining points of $\gamma$ to the origin.

\begin{proposition} \label{prop:cone_example}
    There exists a smooth unit-speed simple closed curve $\gamma : \mathbb{S}^1 \to \mathbb{S}^2$ such that the isometric embedding $u \in W^{2,p}\cap C^{0,1}(B;\R^3) \cap C^\infty(B \setminus \{0\};\R^3)$ ($1 \leq p < 2$) given in polar coordinates by \eqref{eq:u_from_gamma} is the graph of a function $f$ with analogous regularity defined on a star-shaped domain around $0 \in \R^2$ such that 
    \begin{equation*}
        \Det D^2f = c\delta_0 \quad \text{with } c \neq 0.
    \end{equation*}
\end{proposition}

It is more convenient to work with the projection $U'$ of the tentative embedded surface on the $x_1 x_2$ plane. We assume that $U'$ is a star-shaped domain around $0$ given in polar coordinates by
\begin{equation*}
    U' = \big\{ (\tilde r \cos\tilde\theta, \tilde r \sin\tilde\theta) : \tilde\theta \in [0,2\pi], ~\tilde r \in [0,\rho(\tilde\theta)) \big\}, 
\end{equation*}
where $\rho : \mathbb{S}^1 \to (0,1)$ is smooth. In particular, the boundary of $U'$ is a closed simple curve encircling the origin. Introduce now $    z := \rho^{-1}(1-\rho^2)^\frac12$,
so that the embedded surface is the graph of $f : U' \to \mathbb{R}$ defined by $   f(\tilde r \cos\tilde\theta, \tilde r\sin\tilde\theta) := \tilde r z(\tilde\theta)$. It is clear that $f \in W^{2,p} \cap C^{0,1}(B) \cap C^\infty(B \setminus \{0\})$ for any $1 \leq p < 2$. We require that the graph of $f$ over $\partial U'$ has length $2\pi$, that is
\begin{equation} \label{eq:gamma_length_2pi}
      \mathcal{L}[z] := \int_0^{2\pi} \frac {(1+z^2 + \dot{z}^2)^{\frac12}} {1+z^2} d\tilde\theta = 2\pi,
\end{equation}
so that $\gamma$ is the parametrisation by arclength of $f(\partial U')$. Moreover, by a simple calculation, $\det D^2 f = 0$ classically in $B \setminus \{0\}$. This together with \eqref{eq:gamma_length_2pi} corresponds to the statement that the map $u$ defined by \eqref{eq:u_from_gamma} is a smooth isometric embedding of $B \setminus \{0\}$ into $\R^3$.
\begin{lemma} \label{lemma:detD2_cone}
    In the setup above, we have
    \begin{equation} \label{eq:def_c}
        \Det D^2 f = c[z] \delta_0, \quad c[z] = \frac12 \int_0^{2\pi} (z^2 - \dot z^2) d\tilde\theta.
    \end{equation}
\end{lemma}
\begin{proof}
    Recall that $\Det D^2f = \mathrm{div}~V$ in the distributional sense in $B$, where $V := \frac12 \mathrm{cof}(D^2f) Df$. Thus, for any test function $\psi \in C^\infty_c(B)$ we have
    \begin{align*}
         \Det D^2 f[\psi] &= -\lim_{\tilde r \searrow 0} \int_{B \setminus B_{\tilde r}} V \cdot D\psi~dx\\
         &= \lim_{\tilde r \searrow 0} \Big( \int_{B \setminus B_{\tilde r}} \mathrm{div}~V \psi~dx + \int_{\partial B_{\tilde r}} \psi V \cdot \nu ~dS  \Big).
    \end{align*}
    Note that the first term is zero since $\mathrm{div}~V = 0$ classically away from $0$. For the second term, we use that
    $V \cdot \nu = \frac1{2\tilde r}(\ddot z + z)(\tilde\theta) z(\tilde\theta)$, which gives
    \begin{align*}
        \Det D^2 f[\psi] = \frac12 \psi(0) \int_0^{2\pi} (\ddot z + z) z~d\tilde\theta = c[z]\psi(0).
    \end{align*}
\end{proof}

\begin{proof}[Proof of Proposition \ref{prop:cone_example}]
    By Lemma \ref{lemma:detD2_cone}, we only need to show that there exists some $z \in C^\infty(\mathbb{S}^1)$ such that \eqref{eq:gamma_length_2pi} holds and the constant $c[z]$ from \eqref{eq:def_c} is nonzero. More specifically, $z$ will take the form $z(\tilde \theta)  = a z_0(\tilde\theta)$ where $z_0 \in C^\infty(\mathbb{S}^1;(0,\infty))$ is a positive smooth function and $a \in (0,\infty)$ is a constant to be selected.

    A simple computation gives
    \begin{align*}
    \mathcal{L}[0] 
        &= 2\pi,\qquad    \frac{d}{da}\Big|_{a=0} \mathcal{L}[a z_0] 
        = 0, \qquad    \frac{d^2}{da^2}\Big|_{a=0} \mathcal{L}[a z_0] 
        = -2c[z_0].
    \end{align*}
    Moreover, since $\inf_{\mathbb{S}^1} z_0 > 0$, there holds
    \begin{equation*}
        \lim_{a \rightarrow \infty} \mathcal{L}[az_0]  = 0.
    \end{equation*}
    Therefore, provided $c[z_0] < 0$, there exists $a \in (0,\infty)$ such that $\mathcal{L}[az_0] = 2\pi$, that is \eqref{eq:gamma_length_2pi} is satisfied. Hence, as $c[az_0] = a^2 c[z_0]$, we only need to exhibit $z_0 \in C^\infty(\mathbb{S}^1;(0,\infty))$ such that $c[z_0] < 0$. There are many such functions, e.g. $z_0(\tilde\theta) := 1 + \frac{1}{2}\cos (4\tilde\theta)$.
\end{proof}

%%%%%

\subsection*{Rights retention statement} 

For the purpose of Open Access, the authors have applied a CC BY public copyright licence to any Author Accepted Manuscript (AAM) version arising from this submission.

%%%%%

\end{document}